\documentclass[11pt, twoside]{article}
\usepackage{latexsym}
\usepackage{cite}
\usepackage{amsmath}
\usepackage{amssymb}
\usepackage[all]{xy}
\usepackage{amsfonts}
\usepackage{verbatim}
\usepackage{amsthm}
\usepackage{mathrsfs}
\usepackage{epsfig}
\usepackage{xy}
\usepackage{array}
\usepackage{stmaryrd}
\usepackage{graphicx,color}
\usepackage{xcolor}
\usepackage[colorlinks=true,linkcolor=blue,citecolor=blue]{hyperref}
\usepackage{tikz}
\usetikzlibrary{arrows,calc}
\usepackage{etex}
\usepackage{mathdots}
\usepackage{float}
\usepackage{graphics}
\usepackage{pdflscape}
\usepackage{extarrows}
\usepackage{anysize,hyperref}
\input xypic
\xyoption{all}
\usepackage{bm}

\usepackage[perpage,symbol]{footmisc}
\usepackage{setspace}
\usepackage{tikz}

\def\A{\mathcal{A}}

\def\P{\mathcal{P}}
\def\I{\mathcal{I}}

\def\C{\mathcal{C}}
\def\E{\mathbb{E}}

\def\s{\mathfrak{s}}
\def\id{\mathrm{id}}
\def\op{^\mathrm{op}}
\def\Ab{\mathsf{Ab}}
\def\del{\delta}
\def\dr{\ar@{->}[r]}

\def\X{\mathscr{X}}

\def\add{\mbox{add}}

\def\Hom{\mbox{Hom}}

\newcommand{\CC}{{\bf{C}}^{n+2}_{\C}}

\newcommand{\ov}{\overset}
\newcommand{\lra}{\longrightarrow}
\newcommand{\co}{\colon}
\newcommand{\uas}{^{\ast}}            
\newcommand{\sas}{_{\ast}}
\newcommand{\Xd}{\langle X_{\bullet},\del\rangle}  

\newcommand{\ush}{^\sharp}            
\newcommand{\ssh}{_\sharp}
\SelectTips{cm}{10}

\usepackage{setspace}
\usepackage{enumitem}

\begin{document}
\baselineskip=15pt
\title{\Large{\bf Localization of $n$-exangulated categories \footnotetext{\hspace{-1em} Jian He was supported by the National Natural Science Foundation of China (Grant No. 12171230). Panyue Zhou was supported by the National Natural Science Foundation of China (Grant No. 11901190).} }}
\medskip
\author{Jian He, Jing He and Panyue Zhou}

\date{}

\maketitle
\def\blue{\color{blue}}
\def\red{\color{red}}

\newtheorem{theorem}{Theorem}[section]
\newtheorem{lemma}[theorem]{Lemma}
\newtheorem{corollary}[theorem]{Corollary}
\newtheorem{proposition}[theorem]{Proposition}
\newtheorem{conjecture}{Conjecture}
\theoremstyle{definition}
\newtheorem{definition}[theorem]{Definition}
\newtheorem{question}[theorem]{Question}
\newtheorem{remark}[theorem]{Remark}
\newtheorem{remark*}[]{Remark}
\newtheorem{example}[theorem]{Example}
\newtheorem{example*}[]{Example}
\newtheorem{condition}[theorem]{Condition}
\newtheorem{condition*}[]{Condition}
\newtheorem{construction}[theorem]{Construction}
\newtheorem{construction*}[]{Construction}

\newtheorem{assumption}[theorem]{Assumption}
\newtheorem{assumption*}[]{Assumption}

\baselineskip=17pt
\parindent=0.5cm

\begin{abstract}
\begin{spacing}{1.2}
Nakaoka--Ogawa--Sakai considered the localization of an extriangulated category. This construction unified the Serre quotient of abelian categories and the Verdier quotient of triangulated categories. Recently, Herschend--Liu--Nakaoka defined $n$-exangulated categories as a higher dimensional analogue of extriangulated categories. Let $\C$ be an $n$-exangulated category and $\mathcal{F}$  be a multiplicative system satisfying
mild assumption. In this article, we give a necessary and sufficient condition for the localization of $\C$ be an $n$-exangulated category. This way gives a new class of $n$-exangulated categories which are neither $n$-exact nor $(n+2)$-angulated in general.
Moreover, our result also generalizes work by Nakaoka--Ogawa--Sakai.\\[0.2cm]
\\[0.1cm]
 \end{spacing}
\end{abstract}

\pagestyle{myheadings}
\markboth{\rightline {\scriptsize J. He, J. He and P. Zhou }}
         {\leftline{\scriptsize Localization of $n$-exangulated categories  }}

\section{Introduction}
The notion of extriangulated categories was introduced in \cite{NP}, which can be viewed as a simultaneous generalization of exact categories and triangulated categories.
 The data of such a category is a triplet $(\C,\E,\s)$, where $\C$ is an additive category, $\mathbb{E}\colon \C^{\rm op}\times \C \rightarrow \Ab$ is an additive bifunctor and $\mathfrak{s}$ assigns to each $\delta\in \mathbb{E}(C,A)$ a class of $3$-term sequences with end terms $A$ and $C$ such that certain axioms hold. Recently, Herschend--Liu--Nakaoka \cite{HLN}
introduced the notion of $n$-exangulated categories for any positive integer $n$. It should be noted that the case $n=1$ corresponds to
extriangulated categories. As typical examples we know that $n$-exact categories and $(n+2)$-angulated categories are $n$-exangulated categories, see \cite[Proposition 4.34]{HLN} and \cite[Proposition 4.5]{HLN}. However, there are some other examples of $n$-exangulated categories which are neither $n$-exact nor $(n+2)$-angulated, see \cite{HLN, HLN1, LZ,HZZ2}.

It is well known that the localization can be performed in a satisfactory generality for abelian/triangulated categories. In order to unify the theory of localization, Nakaoka-Ogawa-Sakai \cite{HYA} discussed about localizations of extrianguated categories. More specifically,  Nakaoka-Ogawa-Sakai \cite[Theorem 3.5]{HYA} proved the localization of an extriangulated category by a multiplicative system satisfying mild assumptions can be equipped with a natural, universal structure of an extriangulated category. This construction unified the Serre quotient of abelian categories and the Verdier quotient of triangulated categories. Indeed, they also gave such a construction for a bit wider class of morphisms, so that it covered several other localizations appeared in the literature, such as Rump's localization of exact categories by biresolving subcategories, localizations of extriangulated categories by means of Hovey twin cotorsion pairs, and the localization of exact categories by two-sided admissibly percolating subcategories.
Based on this idea, we prove the main result in the article, which is a higher counterpart of Nakaoka-Ogawa-Sakai's result. Moreover, the construction gives $n$-exangulated categories which are neither
$n$-exact nor $(n+2)$-angulated in general.

Let $(\mathcal{C},\mathbb{E},\mathfrak{s})$  be an $n$-exangulated category. Assume that $\mathcal{N}_{\mathcal{F}}\subseteq \C$ is an additive full subcategory associated to $\mathcal{F}$, and $\overline{\mathcal{F}}$ denotes a set of morphisms in the ideal quotient $\overline\C=\C/[\mathcal{N}_{\mathcal{F}}]$ obtained from $\mathcal{F}$ by taking closure with respect to the composition with isomorphisms in $\overline\C$. Our main result is the following.
\begin{theorem}\rm{ (see Theorem \ref{th} for details)}
Let $\mathcal{F}$ be a set of morphisms in $\C$ containing all isomorphisms and closed by compositions. Also, $\mathcal{F}$ is closed by taking finite direct sums.
Suppose that $\overline{\mathcal{F}}$ satisfies the following conditions in $\overline\C$:
\begin{enumerate}[leftmargin=3.4em]
  \item[(MR1)] $\overline{\mathcal{F}}$ satisfies 2-out-of-3 with respect to compositions.
  \item[(MR2)] $\overline{\mathcal{F}}$ is a multiplicative system.
  \item[(MR3)] Let $$X_{0}\xrightarrow{f_{0}} X_{1} \xrightarrow{f_{1}}X_{2}\xrightarrow{f_{2}} \cdots\xrightarrow{f_{n-1}} X_{n}\xrightarrow{f_{n}} X_{n+1}\overset{\delta}{\dashrightarrow}$$ and $$Y_{0}\xrightarrow{g_{0}} Y_{1} \xrightarrow{g_{1}}Y_{2}\xrightarrow{g_{2}} \cdots\xrightarrow{g_{n-1}} Y_{n}\xrightarrow{g_{n}} Y_{n+1}\overset{\delta'}{\dashrightarrow}$$ be any pair of distinguished $n$-exangles, and let $a\in \C(X_{0},Y_{0})$, $c\in \C(X_{n+1},Y_{n+1})$ be any pair of morphisms
satisfying $a_{\ast}\delta = c^{\ast}\delta'$. If $\overline a$ and $\overline c$ belong to $\overline{\mathcal{F}}$, then there exists $\mathbf{b_i} \in \overline{\mathcal{F}}(X_{i},Y_{i})$ such that $\mathbf{b_1}\circ \overline{f_{0}} = \overline{g_{0}}\circ \overline a$, $\overline {c}\circ \overline{f_{n}} = \overline{g_{n}}\circ \mathbf{b_n}$ and $\mathbf{b_{i+1}}\circ \overline{f_{i}} = \overline{g_{i}}\circ \mathbf{b_i}$, where $i=1,2,\cdots,n-1$.
\end{enumerate}
Then the localization of $\C$ by $\mathcal{F}$ gives a weakly $n$-exangulated category $(\widetilde{\mathcal{C}},\widetilde{\mathbb{E}},\widetilde{\mathfrak{s}})$ equipped
with an exact functor $(Q,\mu):(\mathcal{C},\mathbb{E},\mathfrak{s})\rightarrow (\widetilde{\mathcal{C}},\widetilde{\mathbb{E}},\widetilde{\mathfrak{s}})$ if and only if any distinguished $n$-exangle in $\C$ induces a weak kernel-cokernel sequence in $\widetilde{\mathcal{C}}$. If $(\widetilde{\mathcal{C}},\widetilde{\mathbb{E}},\widetilde{\mathfrak{s}})$ moreover satisfies {\rm(C4)}, then it is an $n$-exangulated category.
\end{theorem}
This article is organized as follows. In Section 2, we review some elementary definitions and facts on (weakly) $n$-exangulated categories. In Section 3, we prove our main result and give an example to explain our main result.

\section{Preliminaries}
Let $\C$ be an additive category which is skeletally small, and $n$ be a positive integer. Suppose that $\C$ is equipped with an additive bifunctor $\E\colon\C\op\times\C\to\Ab$, where $\Ab$ is the category of abelian groups. Next we briefly recall some definitions and basic properties of $n$-exangulated categories from \cite{HLN}. We omit some
details here, but the reader can find them in \cite{HLN}.

{ For any pair of objects $A,C\in\C$, an element $\del\in\E(C,A)$ is called an {\it $\E$-extension} or simply an {\it extension}. We abbreviately express it as $C\overset{\delta}{\dashrightarrow}A$.  We also write such $\del$ as ${}_A\del_C$ when we indicate $A$ and $C$. The zero element ${}_A0_C=0\in\E(C,A)$ is called the {\it split $\E$-extension}. For any pair of $\E$-extensions ${}_A\del_C$ and ${}_{A'}\del{'}_{C'}$, let $\delta\oplus \delta'\in\mathbb{E}(C\oplus C', A\oplus A')$ be the
element corresponding to $(\delta,0,0,{\delta}{'})$ through the natural isomorphism $$\mathbb{E}(C\oplus C', A\oplus A')\simeq\mathbb{E}(C, A)\oplus\mathbb{E}(C, A')
\oplus\mathbb{E}(C', A)\oplus\mathbb{E}(C', A').$$

For any $a\in\C(A,A')$ and $c\in\C(C',C)$,  $\E(C,a)(\del)\in\E(C,A')\ \ \text{and}\ \ \E(c,A)(\del)\in\E(C',A)$ are simply denoted by $a_{\ast}\del$ and $c^{\ast}\del$, respectively.

Let ${}_A\del_C$ and ${}_{A'}\del{'}_{C'}$ be any pair of $\E$-extensions. A {\it morphism} $(a,c)\colon\del\to{\delta}{'}$ of extensions is a pair of morphisms $a\in\C(A,A')$ and $c\in\C(C,C')$ in $\C$, satisfying the equality
$a_{\ast}\del=c^{\ast}{\delta}{'}$.}

\begin{definition}\cite[Definition 2.7]{HLN}
Let $\bf{C}_{\C}$ be the category of complexes in $\C$. As its full subcategory, define $\CC$ to be the category of complexes in $\C$ whose components are zero in the degrees outside of $\{0,1,\ldots,n+1\}$. Namely, an object in $\CC$ is a complex $X_{\bullet}=\{X_i,d^X_i\}$ of the form
\[ X_0\xrightarrow{d^X_0}X_1\xrightarrow{d^X_1}\cdots\xrightarrow{d^X_{n-1}}X_n\xrightarrow{d^X_n}X_{n+1}. \]
We write a morphism $f_{\bullet}\co X_{\bullet}\to Y_{\bullet}$ simply $f_{\bullet}=(f_0,f_1,\ldots,f_{n+1})$, only indicating the terms of degrees $0,\ldots,n+1$.
\end{definition}

\begin{definition}\cite[Definition 2.11]{HLN}
By Yoneda lemma, any extension $\del\in\E(C,A)$ induces natural transformations
\[ \del\ssh\colon\C(-,C)\Rightarrow\E(-,A)\ \ \text{and}\ \ \del\ush\colon\C(A,-)\Rightarrow\E(C,-). \]
For any $X\in\C$, these $(\del\ssh)_X$ and $\del\ush_X$ are given as follows.
\begin{itemize}
\item[\rm(1)] $(\del\ssh)_X\colon\C(X,C)\to\E(X,A)\ :\ f\mapsto f\uas\del$.
\item[\rm (2)] $\del\ush_X\colon\C(A,X)\to\E(C,X)\ :\ g\mapsto g\sas\delta$.
\end{itemize}
We simply denote $(\del\ssh)_X(f)$ and $\del\ush_X(g)$ by $\del\ssh(f)$ and $\del\ush(g)$, respectively.
\end{definition}

 \begin{definition}\cite[Definition 2.9]{HLN}
 Let $\C,\E,n$ be as before. Define a category $\AE:=\AE^{n+2}_{(\C,\E)}$ as follows.
\begin{enumerate}
\item[\rm(1)] A object in $\AE^{n+2}_{(\C,\E)}$ is a pair $\Xd$ of $X_{\bullet}\in\CC$
and $\del\in\E(X_{n+1},X_0)$ satisfying
$$(d_0^X)_{\ast}\del=0~~\textrm{and}~~(d^X_n)^{\ast}\del=0.$$
We call such a pair an $\E$-attached
complex of length $n+2$. We also denote it by
$$X_0\xrightarrow{d_0^X}X_1\xrightarrow{d_1^X}\cdots\xrightarrow{d_{n-2}^X}X_{n-1}
\xrightarrow{d_{n-1}^X}X_n\xrightarrow{d_n^X}X_{n+1}\overset{\delta}{\dashrightarrow}.$$
\item[\rm (2)]  For such pairs $\Xd$ and $\langle Y_{\bullet},\rho\rangle$, a morphism $f_{\bullet}\colon\Xd\to\langle Y_{\bullet},\rho\rangle$ is
defined to be a morphism $f_{\bullet}\in\CC(X_{\bullet},Y_{\bullet})$ satisfying $(f_0)_{\ast}\del=(f_{n+1})^{\ast}\rho$.

We use the same composition and the identities as in $\CC$.
\end{enumerate}
\end{definition}

\begin{definition}\cite[Definition 2.13]{HLN}\label{def1}
An {\it $n$-exangle} is a pair $\Xd$ of $X_{\bullet}\in\CC$
and $\del\in\E(X_{n+1},X_0)$ which satisfies the following conditions.
\begin{enumerate}
\item[\rm (1)] The following sequence of functors $\C\op\to\Ab$ is exact.
$$
\C(-,X_0)\xrightarrow{\C(-,\ d^X_0)}\cdots\xrightarrow{\C(-,\ d^X_n)}\C(-,X_{n+1})\xrightarrow{~\del\ssh~}\E(-,X_0)
$$
\item[\rm (2)] The following sequence of functors $\C\to\Ab$ is exact.
$$
\C(X_{n+1},-)\xrightarrow{\C(d^X_n,\ -)}\cdots\xrightarrow{\C(d^X_0,\ -)}\C(X_0,-)\xrightarrow{~\del\ush~}\E(X_{n+1},-)
$$
\end{enumerate}
In particular any $n$-exangle is an object in $\AE$.
A {\it morphism of $n$-exangles} simply means a morphism in $\AE$. Thus $n$-exangles form a full subcategory of $\AE$.
\end{definition}

\begin{definition}\cite[Definition 2.22]{HLN}\label {111}
Let $\s$ be a correspondence which associates a homotopic equivalence class $\s(\del)=[{}_A{X_{\bullet}}_C]$ to each extension $\del={}_A\del_C$. Such $\s$ is called a {\it realization} of $\E$ if it satisfies the following condition for any $\s(\del)=[X_{\bullet}]$ and any $\s(\rho)=[Y_{\bullet}]$.
\begin{itemize}
\item[{\rm (R0)}] For any morphism of extensions $(a,c)\co\del\to\rho$, there exists a morphism $f_{\bullet}\in\CC(X_{\bullet},Y_{\bullet})$ of the form $f_{\bullet}=(a,f_1,\ldots,f_n,c)$. Such $f_{\bullet}$ is called a {\it lift} of $(a,c)$.
\end{itemize}
In such a case, we simply say that \lq\lq$X_{\bullet}$ realizes $\del$" whenever they satisfy $\s(\del)=[X_{\bullet}]$.
\end{definition}

\begin{definition}\cite[Definition 2.23]{HLN}\label{de2.6}
Let $\E$ and $\CC$ be as before. Assume that we are given a realization $\s$  which associates a homotopic equivalence class $\s(\del)=[{}_A{X_{\bullet}}_C]$ to each extension $\del={}_A\del_C$. We use the following terminology.
\begin{enumerate}
\item[\rm (1)] An $n$-exangle $\Xd$ is called an $\s$-{\it distinguished} $n$-exangle if it satisfies $\s(\del)=[X_{\bullet}]$. We often simply say {\it distinguished $n$-exangle} when $\s$ is clear from the context.
\item[\rm (2)]  An object $X_{\bullet}\in\CC$ is called an {\it $\s$-conflation} or simply a {\it conflation} if it realizes some extension $\del\in\E(X_{n+1},X_0)$.
\item[\rm (3)]  A morphism $f$ in $\C$ is called an {\it $\s$-inflation} or simply an {\it inflation} if it admits some conflation $X_{\bullet}\in\CC$ satisfying $d_0^X=f$.
\item[\rm (4)]  A morphism $g$ in $\C$ is called an {\it $\s$-deflation} or simply a {\it deflation} if it admits some conflation $X_{\bullet}\in\CC$ satisfying $d_n^X=g$.
\end{enumerate}
\end{definition}

\begin{definition}\cite[Definition 2.27]{HLN}
For a morphism $f_{\bullet}\in\CC(X_{\bullet},Y_{\bullet})$ satisfying $f_0=\id_A$ for some $A=X_0=Y_0$, its {\it mapping cone} $M_{_{\bullet}}^f\in\CC$ is defined to be the complex
\[ X_1\xrightarrow{d^{M_f}_0}X_2\oplus Y_1\xrightarrow{d^{M_f}_1}X_3\oplus Y_2\xrightarrow{d^{M_f}_2}\cdots\xrightarrow{d^{M_f}_{n-1}}X_{n+1}\oplus Y_n\xrightarrow{d^{M_f}_n}Y_{n+1} \]
where $d^{M_f}_0=\begin{bmatrix}-d^X_1\\ f_1\end{bmatrix},$
$d^{M_f}_i=\begin{bmatrix}-d^X_{i+1}&0\\ f_{i+1}&d^Y_i\end{bmatrix}\ (1\le i\le n-1),$
$d^{M_f}_n=\begin{bmatrix}f_{n+1}&d^Y_n\end{bmatrix}$.

{\it The mapping cocone} is defined dually, for morphisms $h_{\bullet}$ in $\CC$ satisfying $h_{n+1}=\id$.
\end{definition}

\begin{definition}\cite[Definition 2.32]{HLN}
Let $(\C,\E,\s)$ be a triplet of an additive category $\C$, an additive bifunctor $\E\co\C\op\times\C\to\Ab$, and a realization $\s$ which associates a homotopic equivalence class $\s(\del)=[{}_A{X_{\bullet}}_C]$ to each extension $\del={}_A\del_C$.

\begin{itemize}
\item[{\rm (1)}] $(\C,\E,\s)$ is a  {\it weakly $n$-exangulated category} if it satisfies the following conditions:
\end{itemize}

{\rm (C1)} For any $\s(\del)=[X_{\bullet}]$, the pair $\Xd$ is an $n$-exangle.

{\rm (C2)} For any $A\in\C$, the zero element ${}_A0_0=0\in\E(0,A)$ satisfies
\[ \s({}_A0_0)=[A\ov{\id_A}{\lra}A\to0\to\cdots\to0\to0]. \]

{\rm (C2')} For any $A\in\C$, the zero element ${}_00_A\in\E(A,0)$ satisfies
\[ \s({}_00_A)=[0\to0\to0\to\cdots\to A\ov{\id_A}{\lra}A]. \]

{\rm (C3)} For $\delta\in\E(C,A)$ and $a\in\C(A,B)$, let ${}_B\langle Y_{\bullet},a_{\ast}\delta\rangle_C$ and ${}_A\langle X_{\bullet},\delta\rangle_C$ be distinguished $n$-exangles. Then $(a,\id_C)$ has a {\it good lift} $f_{\bullet}$, in the sense that its mapping cone gives a distinguished $n$-exangle $\langle M^f_{\bullet},(d^Y_n)^{\ast}\delta\rangle$.

{\rm (C3')} Dual of {\rm (C3)}.

\begin{itemize}
\item[{\rm (2)}] $(\C,\E,\s)$ is an {\it $n$-exangulated category} if it is weakly $n$-exangulated category and moreover satisfies the following conditions:
\end{itemize}

{\rm (C4)} Let $A\ov{f}{\lra}B\ov{g}{\lra}C$ be any sequence of morphisms in $\C$. If both $f$ and $g$ are inflations, then so is $g\circ f$. Dually, if $f$ and $g$ are deflations, then so is $g\circ f$.
\end{definition}

\begin{remark}(1) For a (weakly) $n$-exangulated category $(\C,\E,\s)$, we call $\s$ an exact realization of $\E$.

(2) In fact, it is easy to see that the condition (R0) in Definition \ref {111} is negligible since it follows from {\rm (C3)} and {\rm (C3')}.

\end{remark}

\begin{proof} Let $\s(\del)=[X_{\bullet}]$ and $\s(\rho)=[Y_{\bullet}]$ satisfying $a_{\ast}\del=c^{\ast}\rho$. By {\rm (C3)} and {\rm (C3')}, we have the following commutative diagram
$$\xymatrix{
X_{\bullet}& X_0\ar[r]^{}\ar@{}[dr] \ar[d]^{a} &X_1 \ar[r]^{} \ar@{}[dr]\ar@{-->}[d]^{\varphi_1}&X_2 \ar[r]^{} \ar@{}[dr]\ar@{-->}[d]^{\varphi_2}&\cdot\cdot\cdot \ar[r]\ar@{}[dr] &X_{n-1} \ar[r]^{} \ar@{}[dr]\ar@{-->}[d]^{\varphi_{n-1} }&X_{n}\ar[r]^{}\ar@{}[dr]\ar@{-->}[d]^{\varphi_{n}}&{X_{n+1}} \ar@{}[dr]\ar@{=}[d]^{} \ar@{-->}[r]^-{\del} &\\& Y_0\ar[r]^{}\ar@{}[dr] \ar@{=}[d]^{} &Z_1 \ar[r]^{} \ar@{}[dr]\ar@{-->}[d]^{\psi_1}&Z_2 \ar[r]^{} \ar@{}[dr]\ar@{-->}[d]^{\psi_2}&\cdot\cdot\cdot \ar[r]\ar@{}[dr] &Z_{n-1} \ar[r]^{} \ar@{}[dr]\ar@{-->}[d]^{\psi_{n-1}}&Z_n\ar[r]^{} \ar@{}[dr]\ar@{-->}[d]^{\psi_{n}}&{X_{n+1}} \ar@{}[dr]\ar[d]^{c} \ar@{-->}[r]^-{a_{\ast}\del=c^{\ast}\rho} &\\
Y_{\bullet}& Y_0 \ar[r]^{}&Y_1 \ar[r]^{} & Y_2 \ar[r]^{}  & \cdots \ar[r]^{} & {Y_{n-1}}\ar[r]^{}  &{Y_{n}}\ar[r]^{}  &{Y_{n+1}} \ar@{-->}[r]^-{\rho} &}
$$
of distinguished $n$-exangles. So $f_{\bullet}=(a,\psi_1\varphi_1,\ldots,\psi_n\varphi_n,c)$ is a {\it lift} of $(a,c)$.

\end{proof}

\begin{remark}(1) Note that the case $n=1$, the definition of weakly $1$-exangulated category coincides with the definition of weakly extriangulated category in \cite[Definition 5.15]{BBGH}. Moreover, a triplet $(\C,\E,\s)$ is a  $1$-exangulated category if and only if it is an extriangulated category, see \cite[Proposition 4.3]{HLN}.

(2) From \cite[Proposition 4.34]{HLN} and \cite[Proposition 4.5]{HLN},  we know that $(n+2)$-angulated in the sense of Geiss--Keller--Oppermann \cite{GKO} and $n$-exact categories in the sense of Jasso \cite{Ja} are $n$-exangulated categories.
There are some other examples of $n$-exangulated categories
 which are neither $n$-exact nor $(n+2)$-angulated, see \cite{HLN,HLN1,LZ,HZZ2}.
\end{remark}

In this article, we always assume that the following condition, analogous to the (WIC) Condition
in \cite[Condition 5.8]{NP}.
\begin{condition}\label{cd}
Let $f \in \C(A,B)$, $g \in\C(B,C)$ be any composable pair of morphisms. For a (weakly) $n$-exangulated category $(\C,\E,\s)$, consider the following
conditions.

(1) If $g \circ f$ is a deflation, then so is $g$.

(2) If $g \circ f$ is an inflation, then so is $f$.
\end{condition}

\begin{definition}\label{efu} Let $(\C,\E,\s)$, $(\C',\E',\s')$ and $(\C'',\E'',\s'')$ be weakly $n$-exangulated categories.

(1) An additive covariant functor $F:\C\rightarrow\C' $ is called an exact functor if there exists a natural transformation $\Gamma:\mathbb{E}\Longrightarrow\E'\circ (F^{\rm op}\times F )$ of functors $\C^{\rm op}\times \C \rightarrow \Ab$ such that $\s (\delta)=[X_\bullet]$ implies $\s'(\Gamma_{(X_{n+1},X_{0})}(\delta))=[FX_\bullet]$.

(2) If $(F,\phi):(\C,\E,\s)\rightarrow(\C',\E',\s')$ and $(F',\phi'):(\C',\E',\s')\rightarrow(\C'',\E'',\s'')$ are exact functor, then their composition $(F'',\phi'')=(F',\phi')\circ(F,\phi)$ is defined by $F''=F'F$ and $\phi''=(\phi'\circ(F^{\rm op}\times F)).\phi$.

(3) Let $(F,\phi),(G,\psi):(\C,\E,\s)\rightarrow(\C',\E',\s')$ be exact functor. A natural transformation $\eta:(F,\phi)\Longrightarrow(G,\psi)$ of exact functor is a natural transformation $\eta:F\Longrightarrow G$ of additive functors, which satisfies $$(\eta_{X_0})_{\ast}\phi_{X_{n+1},X_0}(\del)=(\eta_{X_{n+1}})^{\ast}\psi_{X_{n+1},X_0}(\del)$$
for any $\delta\in \E(X_{n+1},X_0)$. Horizontal compositions and vertical compositions are defined by those for natural transformations of additive functors.
\end{definition}
\begin{remark}
(1) The above definition of an exact functor in (1) is nothing but that of an $n$-exangulated functor introduced in \cite[Definition 2.32]{B-TS}, applied to weakly $n$-exangulated categories.
\medskip

(2) When $n=1$, Definition \ref{efu} is just Definition 2.11 in \cite{HYA}.
\end{remark}

\begin{proposition}\label{rrr}Let $(F,\phi):(\C,\E,\s)\rightarrow(\C',\E',\s')$ be an exact functor between weakly $n$-exangulated categories. The following statements are equivalent.

{\rm (1)} $F$ is an equivalence of categories and $\phi$ is a natural isomorphism.

{\rm (2)} $(F,\phi)$ is an equivalence of $n$-exangulated categories in the sense that there exist an exact functor $(G,\psi):(\C',\E',\s')\rightarrow(\C,\E,\s)$, natural transformations of exact functors $(\id_\C, \id_\E)\Longrightarrow (G,\psi)\circ(F,\phi)$ and $(F,\phi)\circ(G,\psi)\Longrightarrow (\id_{\C'}, \id_{\E'})$ which have inverses.
\end{proposition}
\begin{proof}
It is similar to \cite[Proposition 2.13]{HYA}, we omit it.
\end{proof}

\section{Localization }
Let $\C$ be an additive category. We denote by $\mathcal{M}_{\C}$ the class of all morphisms in $\C$. If a class of morphisms $\mathcal{F}\subseteq \mathcal{M}_{\C}$
is closed by compositions and contains all identities in $\A$, then we may regard
$\mathcal{F}\subseteq\C$ as a (not full) subcategory satisfying $Ob(\mathcal{F}) = Ob(\C)$. With this view in
mind, we write $\mathcal{F}(X,Y) = \{f \in \C(X,Y) | f \in \mathcal{F}\}$ for any $X$, $Y\in\C$.  We denote by $\textrm{Iso}~ \C$ the class of all isomorphisms in $\C$. $\mathcal{N}_{\mathcal{F}}$ denotes the full subcategory consisting of objects $N\in \C$ such that both $N\rightarrow 0$ and $0\rightarrow N$ belong to $\mathcal{F}$. It is obvious that $\mathcal{N}_{\mathcal{F}}\subseteq \C$ is an additive subcategory. In the rest, we will denote the ideal quotient by $p:{\C}\rightarrow\overline{\C}=\C/[\mathcal{N}_{\mathcal{F}}]$, and $\overline{f}$ will denote a morphism in $\overline{\C}$ represented by $f\in\C(X,Y)$. $\overline{\mathcal{F}}$ denotes a set of morphisms in the ideal quotient
$\C/[\mathcal{N}_{\mathcal{F}}]$ obtained from $\mathcal{F}$ by taking closure with respect to the composition with isomorphisms in $\overline\C$. \par
Now, let $(\mathcal{C},\mathbb{E},\mathfrak{s})$  be an $n$-exangulated category and  $\mathcal{F}\subseteq \mathcal{M}_{\C}$  a set of morphisms which satisfies the following conditions.
\begin{enumerate}[leftmargin=3.3em]
  \item[(M0)] $\mathcal{F}$ contains all isomorphisms in $\C$, and closed by compositions. Also, $\mathcal{F}$ is closed by taking finite direct sums. Namely, if $f_{i}\in\mathcal{F}(X_{i},Y_{i})$, for $i=1,2$, then $f_{1}\oplus f_{2}\in\mathcal{F}(X_{1}\oplus X_{2},Y_{1}\oplus Y_{2})$.
  \item[(MR1)] $\overline{\mathcal{F}}$ satisfies 2-out-of-3 with respect to compositions.
  \item[(MR2)] $\overline{\mathcal{F}}$ is a multiplicative system.
  \item[(MR3)] Let $$X_{0}\xrightarrow{f_{0}} X_{1} \xrightarrow{f_{1}}X_{2}\xrightarrow{f_{2}} \cdots\xrightarrow{f_{n-1}} X_{n}\xrightarrow{f_{n}} X_{n+1}\overset{\delta}{\dashrightarrow}$$ and $$Y_{0}\xrightarrow{g_{0}} Y_{1} \xrightarrow{g_{1}}Y_{2}\xrightarrow{g_{2}} \cdots\xrightarrow{g_{n-1}} Y_{n}\xrightarrow{g_{n}} Y_{n+1}\overset{\delta'}{\dashrightarrow}$$ be any pair of distinguished $n$-exangles, and let $a\in \C(X_{0},Y_{0})$, $c\in \C(X_{n+1},Y_{n+1})$ be any pair of morphisms
satisfying $a_{\ast}\delta = c^{\ast}\delta'$. If $\overline a$ and $\overline c$ belong to $\overline{\mathcal{F}}$, then there exists $\mathbf{b_i} \in \overline{\mathcal{F}}(X_{i},Y_{i})$ such that $\mathbf{b_1}\circ \overline{f_{0}} = \overline{g_{0}}\circ \overline a$, $\overline {c}\circ \overline{f_{n}} = \overline{g_{n}}\circ \mathbf{b_n}$ and $\mathbf{b_{i+1}}\circ \overline{f_{i}} = \overline{g_{i}}\circ \mathbf{b_i}$, where $i=1,2,\cdots,n-1$.
\end{enumerate}
\begin{lemma}\rm \cite[Lemma 3.2]{HYA}\label{LemS}
Let $\mathcal{F}$ and $\overline{\C}$ be as above. The following statements hold:

(1) Suppose that $f,g\in\C(A,B)$ satisfy $\overline{f}=\overline{g}$ in $\overline{\C}$. Then $f\in\mathcal{F}$ holds if and only if $g\in\mathcal{F}$.

(2) The following conditions are equivalent.
\begin{itemize}
\item[{\rm (i)}] $p (\mathcal{F})=\overline{\mathcal{F}}$.
\item[{\rm (ii)}] $p^{-1}(\overline{\mathcal{F}})=\mathcal{F}$.
\item[{\rm (iii)}] $p^{-1}(\textrm{Iso}(\overline{\C}))\subseteq \mathcal{F}$.
\item[{\rm (iv)}]  $f\in\mathcal{F}$ holds for any split monomorphism $f\in\C(A,B)$ such that $\overline{f}$ is an isomorphism in $\overline{\C}$.
\item[{\rm (v)}] $f\in\mathcal{F}$  holds for any split epimorphism $f\in\C(A,B)$ such that $\overline{f}$ is an isomorphism in $\overline{\C}$.
\end{itemize}

\begin{remark}
If $\C$ is weakly idempotent complete, then any $\mathcal{F}$ with {\rm (M0)} satisfies {\rm (iv)} and hence all the other equivalent conditions. In particular if $(\mathcal{C},\mathbb{E},\mathfrak{s})$ satisfies Condition \ref{cd}, then any $\mathcal{F}$ with {\rm (M0)} should satisfy $p (\mathcal{F})=\overline{\mathcal{F}}$.
\end{remark}

\end{lemma}
Suppose that $\mathcal{F}=p^{-1}(\overline{\mathcal{F}})$.
Notice that a localization $\C\to\widetilde{\mathcal{C}}$ factors uniquely through $p:\C\to\overline{\mathcal{C}}$, and we may regard $\widetilde{\mathcal{C}}$ as a localization of $\overline{\mathcal{C}}$ by $\overline{\mathcal{F}}$.  In this view,
 the objects of $\widetilde{\mathcal{C}}$ is same as $\overline{\C}$ and the morphism set $\Hom_{\widetilde{\mathcal{C}}}(\overline{X}, \overline{Y})$  is a set of  equivalence classes of right roofs which are the diagrams of the form
 $$\xymatrix{&Z\ar@{=>}[dl]_{\overline{s}}\ar[dr]^{\overline{f}}&\\
 \overline{X}&&\overline{Y}    }$$
  or simply written as the right fractions  $\overline{f}/\overline{s}$, where $\overline{s}\in\overline{\mathcal{F}}$ and $\overline{f}\in {\overline{\C}}(\overline{Z},\overline{Y})$.
\begin{definition}
Take a localization $\overline{Q}:\overline{\C}\rightarrow\widetilde{\C}$ of $\overline{\C}$ by the multiplicative system $\overline{\mathcal F}$, and put $Q=\overline{Q}\circ p:\C\rightarrow\widetilde{\C}$. This gives a localization of $\C$ by ${\mathcal F}$.
\end{definition}

Let $(\C, \mathbb{E}, \mathfrak{s})$ be an $n$-exangulated category and $\widetilde{\mathcal{C}}$ as a localization of ${\mathcal{C}}$.
Assume that
$$A_0\xrightarrow{~\alpha_0~}A_1\xrightarrow{~\alpha_1~}A_2\xrightarrow{~\alpha_2~}\cdots\xrightarrow{~\alpha_{n-2}~}A_{n-1}
\xrightarrow{~\alpha_{n-1}~}A_n\xrightarrow{~\alpha_n~}A_{n+1}\overset{\delta}{\dashrightarrow}$$
is a distinguished $n$-exangle in $\C$. This sequence
$$A_0\xrightarrow{{Q(\alpha_0)}}A_1\xrightarrow{Q(\alpha_1)}A_2\xrightarrow{Q(\alpha_2)}\cdots
\xrightarrow{Q(\alpha_{n-2})}A_{n-1}
\xrightarrow{Q(\alpha_{n-1})}A_n\xrightarrow{Q(\alpha_n)}A_{n+1}$$
is called \emph{weak kernel-cokernel sequence} in $\widetilde{\mathcal{C}}$ if the following sequences
$$
\widetilde{\C}(-,A_0)\xrightarrow{\widetilde{\C}(-,\  Q(\alpha_0))}\widetilde{\C}(-,A_1)\xrightarrow{\widetilde{\C}(-,\ Q(\alpha_1))}\cdots\xrightarrow{\widetilde{\C}(-,\ Q(\alpha_{n-1}))}\widetilde{\C}(-,A_n)\xrightarrow{\widetilde{\C}(-,\ Q(\alpha_n))}\widetilde{\C}(-,A_{n+1})
$$
and
$$
\widetilde{\C}(A_{n+1},-)\xrightarrow{\widetilde{\C}(Q({\alpha_n}),\ -)}\widetilde{\C}(A_{n},-)\xrightarrow{\widetilde{\C}(Q({\alpha_{n-1}}),\ -)}\cdots\xrightarrow{\widetilde{\C}(Q({\alpha_{1}}),\ -)}\widetilde{\C}(A_1,-)\xrightarrow{\widetilde{\C}(Q({\alpha_{0}}),\ -)}\widetilde{\C}(A_0,-)
$$
are exact.

Our main result is the following.
\begin{theorem}\label{th}
Let $(\mathcal{C},\mathbb{E},\mathfrak{s})$  be an $n$-exangulated category and  $\mathcal{F}\subseteq \mathcal{M}_{\C}$ satisfies {\rm(M0)}.

{\rm (1)} Suppose that $\overline{\mathcal{F}}$ satisfies {\rm(MR1)}, {\rm(MR2)} and {\rm(MR3)}, then the localization of $\C$ by $\mathcal{F}$ gives a weakly $n$-exangulated category $(\widetilde{\mathcal{C}},\widetilde{\mathbb{E}},\widetilde{\mathfrak{s}})$ equipped
with an exact functor $$(Q,\mu):(\mathcal{C},\mathbb{E},\mathfrak{s})\rightarrow (\widetilde{\mathcal{C}},\widetilde{\mathbb{E}},\widetilde{\mathfrak{s}})$$ if and only if any distinguished $n$-exangle in $\C$ induces a weak kernel-cokernel sequence in $\widetilde{\mathcal{C}}$.

{\rm (2)} If $(\widetilde{\mathcal{C}},\widetilde{\mathbb{E}},\widetilde{\mathfrak{s}})$ moreover satisfies {\rm(C4)}, then it is an $n$-exangulated category if and only if any distinguished $n$-exangle in $\C$ induces a weak kernel-cokernel sequence in $\widetilde{\mathcal{C}}$.

\end{theorem}
{\bf In order to prove Theorem \ref{th}, we need some preparations as follows.}

\subsection{Construction of $\widetilde{\mathbb{E}}$}
In this section, our goal is shown in the following diagram

$$\begin{tikzpicture}
\draw(-0.7,0.5) node{$\circlearrowleft$};
\draw(-1.5,-0.5) node{$\mathop{\Longrightarrow}\limits^{\wp}$};
\draw(0.2,-0.5) node{$\mathop{\Longrightarrow}\limits^{\overline\mu}$};
\draw(0,0) node{$
\xymatrix@C=1.3cm{{\C^{\rm op}\times \C}\ar@/^3pc/[rr]^{Q^{\rm op}\times Q}\ar[r]^{p^{\rm op}\times p}\ar[dr]_{{\E}}&{\overline{\C}^{\rm op}\times \overline{\C}}\ar[r]^{\overline{Q}^{\rm op}\times \overline{Q}}\ar[d]_{\overline{\E}}&{\widetilde{\C}^{\rm op}\times \widetilde{\C}.} \ar[dl]^{\widetilde{\E}}\\
&\Ab&&}
$};
\end{tikzpicture}$$

In fact, in order to construct an additive bifunctor $\widetilde{\E}\colon\widetilde{\C}^{\rm op}\times \widetilde{\C}\to{\rm Ab}$ and a natural transformation $\mu:\mathbb{E}\Longrightarrow\widetilde{\E}\circ (Q^{\rm op}\times Q )$, we will define an additive bifunctor $\overline{\E}$ and natural transformation $\overline\mu, \wp$ as shown in the figure above. And then we can define $\mu=(\overline\mu\circ(p^{\rm op}\times p)).\wp$.

\begin{lemma}\rm \label{LemE}

For any extension $\delta\in\mathbb{E}(A_{n+1},A_0)$, the following are equivalent.

(1) There exists a morphism $s\in\mathcal{F}(A_0,B_0)$ such that $s\sas\delta=0$.

(2) There exists a morphism $t\in\mathcal{F}(B_{n+1},A_{n+1})$ such that $t\uas\del=0$.

\end{lemma}

\begin{proof}
We only prove that $(1)\Rightarrow (2)$, dually one can prove $(2)\Rightarrow (1)$.
Suppose that $s\in\mathcal{F}(A_0,B_0)$ such that $s\sas\delta=0$. Let
$$A_{0}\xrightarrow{x_{0}} A_{1} \xrightarrow{x_{1}}A_{2}\xrightarrow{x_{2}} \cdots\xrightarrow{x_{n-1}} A_{n}\xrightarrow{x_{n}} A_{n+1}\overset{\delta}{\dashrightarrow}$$ and
$$B_{0}\xrightarrow{y_{0}} B_{1} \xrightarrow{y_{1}}B_{2}\xrightarrow{y_{2}} \cdots\xrightarrow{y_{n-1}} B_{n}\xrightarrow{y_{n}} A_{n+1}\overset{s\sas\delta}{\dashrightarrow}$$
be two distinguished $n$-exangles, by \cite[Lemma 3.3 ]{ZW}, we have $\id_{A_{n+1}}$ factors through $y_n$, i.e., there exists a morphism $m\in\C(A_{n+1},B_{n})$, such that $y_{n}m=\id_{A_{n+1}}$. Notice that the equivalent conditions in Lemma \ref {LemS} (2) are satisfied by our assumption. By (MR3), there exists $s_{i}\in\mathcal{F}(A_i,B_i)$ with $i=1,2,\cdots,n$, which makes the following diagram
$$\xymatrix{A_{0}\ar[r]^{\overline {x_0}}\ar@{}[dr] \ar[d]^{\overline{s}} &A_1 \ar[r]^{\overline {x_1}} \ar@{}[dr]\ar[d]^{\overline{s_1}}&\cdot\cdot\cdot \ar[r]^{\overline {x_{n-2}}} \ar@{}[dr]&A_{n-1} \ar[r]^{\overline{x_{n-1}}}\ar@{}[dr]\ar[d]^{\overline{s_{n-1}}} &A_n \ar[r]^{\overline{x_n}} \ar@{}[dr]\ar[d]^{\overline{s_n}}&A_{n+1}\ar@{}[dr]\ar@{=}[d]^{}\ar[dl]_{ \overline{m}}^{\circlearrowleft} &\\
{B_{0}}\ar[r]^{\overline{y_0}} &{B_1}\ar[r]^{\overline{y_1}}&\cdot\cdot\cdot\ar[r]^{\overline{y_{n-2}}} &{B_{n-1}}  \ar[r]^{\overline{y_{n-1}}} &{B_n}\ar[r]^{\overline{y_n}}  &{A_{n+1}} &}
$$
commutative in $\overline{\C}$. By (MR2), we obtain a commutative square
$$\xymatrix{{\clubsuit}\ar@{}[dr]|-{\circlearrowleft}\ar@{=>}[r]^{\overline{k}}\ar[d]_{\overline{h}}&{A_{n+1}}\ar[d]_{\overline{m}}\\
A_n\ar@{=>}[r]^{\overline{s_{n}}}&B_n}$$
in $\overline{\C}$ such that $k\in\mathcal{F}$. Then we have $\overline{k}=\overline{y_n\circ s_n \circ h}=\overline{x_n\circ h}$, which forces $x_n\circ h\in\mathcal{F}$ by Lemma \ref {LemS} (1). Thus $t=x_n\circ h$ satisfies the required properties since $(x_n\circ h){\uas}\delta=h{\uas}({x_n}{\uas}\delta)=0$.
\end{proof}

\begin{construction}\label{cos}
For each $A_0,A_{n+1}\in\C$, define a subset $\mathcal{K}(A_{n+1},A_0)\subseteq\mathbb{E}(A_{n+1},A_0)$ by
\begin{eqnarray*}
\mathcal{K}(A_{n+1},A_0)&=&\{\delta\in \mathbb{E}(A_{n+1},A_0)\mid s\sas\delta=0\ \text{for some}\ s\in\mathcal{F}(A_0,B_0)\}\\
&=&\{\delta\in\mathbb{E}(A_{n+1},A_0)\mid t\uas\delta=0\ \text{for some}\ t\in\mathcal{F}(B_{n+1},A_{n+1})\}.
\end{eqnarray*}

\end{construction}

\begin{remark}\label{rem1}
By Lemma~\ref{LemE}, we know that $\mathcal{K}$ form  a subfunctor of $ \mathbb{E}$. Moreover, $\mathcal{K}\subseteq \mathbb{E}$ is an additive subfunctor. Indeed, for any $\delta\in\mathcal{K}(A_{n+1},A_0)$ and $\delta'\in\mathcal{K}(A'_{n+1},A'_0)$, we have $\delta\oplus\delta'\in\mathcal{K}(A_{n+1}\oplus A'_{n+1},A_0\oplus A'_0)$. Here $\delta\oplus\delta'\in\mathbb{E}(A_{n+1}\oplus A'_{n+1},A_0\oplus A'_0)$ is the element corresponding to $(\delta,0,0,\delta')$ through the isomorphism $\mathbb{E}(A_{n+1}\oplus A'_{n+1},A_0\oplus A'_0)\cong\mathbb{E}(A_{n+1},A_0)\oplus\mathbb{E}(A_{n+1},A'_0)\oplus\mathbb{E}(A'_{n+1},A_0)\oplus\mathbb{E}(A'_{n+1},A'_0)$ induced by the biadditivity of $\mathbb{E}$. This shows the additivity of $\mathcal{K}\subseteq\mathbb{E}$.
\end{remark}

\begin{proposition}\label{pro1}
$\mathbb{E}/\mathcal{K}\colon\C\op\times\C\to{\rm Ab}$ is an additive bifunctor, which satisfies $(\mathbb{E}/\mathcal{K})(N,-)=0$ and $(\mathbb{E}/\mathcal{K})(-,N)=0$ for any $N\in\mathcal{N}_{\mathcal{F}}$.
\end{proposition}
\begin{proof}
It follows from Remark \ref{rem1} and the definition.
\end{proof}
In the rest, for any $\delta\in\mathbb{E}(X,Y)$, let $\overline{\delta}\in(\mathbb{E}/\mathcal{K})(X,Y)$ denote an element represented by $\delta$.
 By Proposition \ref{pro1}, we get a well-defined additive bifunctor $\overline{\mathbb{E}}\colon\overline{\C}\op\times\overline{\C}\to{\rm Ab}$ as follows:

$\bullet$ $\overline{\mathbb{E}}(A_{n+1},A_0)=\mathbb{E}/\mathcal{K}(A_{n+1},A_0)$ for any $A_0,A_{n+1}\in\C$.

$\bullet$ $\overline{a}\sas\overline{\delta}=\overline{a\sas\delta}$ for any $\overline{\delta}\in\overline{\mathbb{E}}(A_{n+1},A_0)$ and $\overline{a}\in\overline{\C}(A_0,A'_0)$.

$\bullet$ $\overline{c}\uas\overline{\delta}=\overline{c\uas\delta}$ for any $\overline{\delta}\in\overline{\mathbb{E}}(A_{n+1},A_0)$ and $\overline{a}\in\overline{\C}(A'_{n+1},A_{n+1})$.

Furthermore, we have a natural transformation $\wp\colon\mathbb{E}\Rightarrow\overline{\mathbb{E}}\circ({p}\op\times p)$ given by
$$\wp_{A_{n+1},A_0}\colon\mathbb{E}(A_{n+1},A_0)\rightarrow\overline{\mathbb{E}}(A_{n+1},A_0); ~~\delta\mapsto\overline{\delta}$$
for each $A_{n+1},A_0\in\C$.

\begin{remark}\label{rem2} The following statements are equivalent for any $\delta\in \mathbb{E}(A_{n+1},A_0)$ by definition.

(1) $\overline{\delta}=0$ holds in $\overline{\mathbb{E}}(A_{n+1},A_0)$.

(2) $\overline{s}\sas\overline{\delta}=0$ holds in $\overline{\mathbb{E}}(A_{n+1},A'_0)$ for some/any $\overline{s}\in\overline{\mathcal{F}}(A_0,A'_0)$.

(3) $\overline{t}\uas\overline{\delta}=0$ holds in $\overline{\mathbb{E}}(A'_{n+1},A_0)$ for some/any $\overline{s}\in\overline{\mathcal{F}}(A'_{n+1},A_{n+1})$.

\end{remark}

 Next, we will construct an additive bifunctor $\widetilde{\mathbb{E}}\colon\widetilde{\C}\op\times\widetilde{\C}\to{\rm Ab}$. To be more specific, we will define $\widetilde{\mathbb{E}}(C,A)$ to be a set of equivalence classes of triplets $(C\ov{\overline t}{\Longleftarrow}X_{n+1}\ov{\overline \delta}{\dashrightarrow}X_{0}\ov{\overline s}{\Longleftarrow}A)$ with respect to some equivalence relation.

By {\rm (MR1)} and {\rm (MR2)}, we have the following key lemma.
\begin{lemma}\label{LemCom}The following statements are true.

{\rm (1)} For any finite family of morphisms $\{\overline{s_i}\in\overline{\mathcal{F}}(A,X_i)\}_{1\le i\le n}$, there exist $X\in\C$, $\overline{s}\in\overline{\mathcal{F}}(A,X)$ and $\{\overline{u_i}\in\overline{\mathcal{F}}(X_i,X)\}_{1\le i\le n}$ such that $\overline{s}=\overline{u_i}\circ\overline{s_i}$ for any $1\le i\le n$.

{\rm (2)} Dual of {\rm (1)}.

\end{lemma}

The following result is essentially contained in \cite[Propsition 3.19 ]{HYA}. However, it can be
extended to our setting. We give a proof for completeness.

\begin{proposition}\label{Prop1}
Let $n\ge2$ be an integer and $A,C\in\C$ be any pair of objects. Suppose that
$$(\overline{ t_{i}}\backslash \overline{ \delta_{i}}/\overline{ s_{i}})=(C\ov{\overline t_{i}}{\Longleftarrow}Z_{i}\ov{\overline \delta}{\dashrightarrow}X_{i}\ov{\overline s_{i}}{\Longleftarrow}A)  ~~~~(i=1,2,\cdots,n )$$
are triplets with $\overline{s_{i}}\in\overline{\mathcal{F}}(A,X_{i})$, $\overline{t_{i}}\in\overline{\mathcal{F}}(Z_{i},C)$, $\overline{ \delta_{i}}\in\overline{\mathbb{E}}(Z_{i},X_{i})$. The following holds.

{\rm (1)}  We can take a \emph{common denominator} $\overline{s},\overline{t}\in\overline{\mathcal{F}}$. That is to say, we have the following commutative diagram for $1\leq i\leq n$,
\begin{equation}
\xymatrix{&{Z_{i}}\ar@{=>}[ld]_{\overline{t_{i}}}\ar@{}[ld]^{\circlearrowleft}\ar@{-->}[r]^{\overline{\delta_{i}}}\ar@{<=}[d]^{\overline{v_{i}}}&{X_{i}}\ar@{}[rd]_{\circlearrowleft}\ar@{<=}[rd]^{\overline{s_{i}}}\ar@{=>}[d]_{\overline{u_{i}}}&\\
C\ar@{<=}[r]_{\overline{t}}&Z\ar@{-->}[r]_{\overline{\rho_{i}}}&X\ar@{<=}[r]_{\overline{s}} &A}
\end{equation}
where we take $\overline{\rho_{i}}=\overline{v_{i}}\uas\overline{u_{i}}\sas\overline{\delta_{i}}$. When $n=2$, we write $(\overline{ t_{1}}\backslash \overline{ \delta_{1}}/\overline{ s_{1}})\sim(\overline{ t_{2}}\backslash \overline{ \delta_{2}}/\overline{ s_{2}})$ if they satisfy $\overline{\rho_{1}}=\overline{\rho_{2}}$.

{\rm (2)} The relation $(\overline{ t_{1}}\backslash \overline{ \delta_{1}}/\overline{ s_{1}})\sim(\overline{ t_{2}}\backslash \overline{ \delta_{2}}/\overline{ s_{2}})$ is independent of the choice of common denominators. Namely, for any other commutative diagram with $i=1,2$,
\begin{equation}\xymatrix{&{Z_{i}}\ar@{=>}[ld]_{\overline{t_{i}}}\ar@{}[ld]^{\circlearrowleft}\ar@{-->}[r]^{\overline{\delta_{i}}}\ar@{<=}[d]^{\overline{v'_{i}}}&{X_{i}}\ar@{}[rd]_{\circlearrowleft}\ar@{<=}[rd]^{\overline{s_{i}}}\ar@{=>}[d]_{\overline{u'_{i}}}&\\
C\ar@{<=}[r]_{\overline{t'}}&Z'\ar@{-->}[r]_{\overline{\rho'_{i}}}&X'\ar@{<=}[r]_{\overline{s'}} &A}\end{equation}
satisfying $\overline{\rho'_{i}}=\overline{v'_{i}}\uas\overline{u'_{i}}\sas\overline{\delta_{i}}$, we have $\overline{\rho_{1}}=\overline{\rho_{2}}$ if and only if $\overline{\rho'_{1}}=\overline{\rho'_{2}}$.
\end{proposition}

\begin{proof}
(1) It follows from Lemma \ref{LemCom}.

(2) Keep the notations in the diagram (3.1) and (3.2), we only prove that if $\overline{\rho_{1}}=\overline{\rho_{2}}$ then $\overline{\rho'_{1}}=\overline{\rho'_{2}}$. Consider triplets
$(\overline{ t}\backslash \overline{ \rho_{i}}/\overline{ s})$ and $(\overline{ t'}\backslash \overline{ \rho'_{i}}/\overline{ s'})$, by Lemma \ref{LemCom}, the exists a common denominator $\overline{s''}, \overline{t''} \in\overline{\mathcal{F}}$ together with the following commutative diagrams
$$\xymatrix{Z_i\ar@{}[rrd]^{\circlearrowleft}\ar@{<=}[r]^{\overline{v_i}}\ar@{=>}[rrd]_{\overline{t_i}}&Z\ar@{}[dr]^{\circlearrowleft}\ar@{=>}[rd]|{\overline{t}}\ar@{<=}[r]^{\overline{v}}&Z''\ar@{=>}[d]|{\overline{t''}}\ar@{=>}[r]^{\overline{v'}}&{Z'}\ar@{}[dl]_{\circlearrowleft}\ar@{=>}[ld]|{\overline{t'}}\ar@{=>}[r]^{\overline{v'_i}}&{Z_i}\ar@{}[lld]_{\circlearrowleft}\ar@{=>}[lld]^{\overline{t_i}}\\
&&C&&}$$
and
$$\xymatrix{&&A\ar@{}[lld]^{\circlearrowleft}\ar@{}[ld]^{\circlearrowleft}\ar@{}[rd]_{\circlearrowleft}\ar@{}[rrd]_{\circlearrowleft}\ar@{=>}[lld]_{\overline{s_i}}\ar@{=>}[ld]|{\overline{s}}\ar@{=>}[d]|{\overline{s''}}\ar@{=>}[rd]|{\overline{s'}}\ar@{=>}[rrd]^{\overline{s_i'}}&&\\
X_i\ar@{=>}[r]_{\overline{u_i}} &X\ar@{=>}[r]_{\overline{u}} &X''\ar@{<=}[r]_{\overline{u'}}&X'\ar@{<=}[r]_{\overline{u'_i}}&X_i.}$$

Composing some morphisms $Z^{\prime\prime\prime}\to Z''$ and $X''\to X^{\prime\prime\prime}$ in $\overline{\mathcal{F}}$ if necessary, we may assume that $\overline{v_iv}=\overline{v'_iv'}$ and $\overline{uu_i}=\overline{u' u'_i}$ hold for $i=1,2$ from the first.
Then we have
$$\overline{v}\uas\overline{u}_{\ast}\overline{\rho_i} = \overline{v}\uas\overline{u}_{\ast}(\overline{v_i}\uas\overline{u_i}_{\ast}\overline{\delta_i})=(\overline{v_iv})\uas(\overline{uu_i})_\ast\overline{\delta_i}
=(\overline{v'_iv'})\uas(\overline{u'u'_i})_\ast\overline{\delta_i}=\overline{v'}\uas\overline{u'}_{\ast}(\overline{v'_i}\uas\overline{u'_i}_{\ast}\overline{\delta_i})=\overline{v'}\uas\overline{u'}_{\ast}\overline{\rho'_i}. $$
So we have $\overline{v'}\uas\overline{u'}_{\ast}\overline{\rho'_1}=\overline{v'}\uas\overline{u'}_{\ast}\overline{\rho'_2}$ by the assumption. Note that $\overline{v'}\uas\overline{u'}_{\ast}(\overline{\rho'_1}-\overline{\rho'_2})=0$, then we have $\overline{\rho'_{1}}=\overline{\rho'_{2}}$ by Remark \ref{rem2}.
\end{proof}
\begin{remark}\label{rem3}
The above relation $\sim$ gives an equivalence relation on the set of triplets
$$\{(\overline{ t}\backslash \overline{ \delta}/\overline{ s})=(C\ov{\overline t}{\Longleftarrow}Z\ov{\overline \delta}{\dashrightarrow}X\ov{\overline s}{\Longleftarrow}A)|X,Z\in\C,\overline{s},\overline{t}\in\overline{\mathcal{F}},\delta\in\overline{\mathbb{E}}(Z,X)  \}~~~~~~(\maltese)$$
for each $A,C\in\C$.
\end{remark}
\begin{proof} One can easily check reflexivity and symmetry. Transitivity can be shown by using Proposition \ref{Prop1}.
\end{proof}

\begin{construction}\label{cos0}
Let $A,C\in\C$ be any pair of objects. Define $\widetilde{\mathbb{E}}(C,A)$ to be the quotient set of $(\maltese)$ by the equivalence relation $\sim$ obtained above.
We denote the equivalence class of $(\overline{ t}\backslash \overline{ \delta}/\overline{ s})=(C\ov{\overline t}{\Longleftarrow}Z\ov{\overline \delta}{\dashrightarrow}X\ov{\overline s}{\Longleftarrow}A) $ by $[\overline{ t}\backslash \overline{ \delta}/\overline{ s}]=[C\ov{\overline t}{\Longleftarrow}Z\ov{\overline \delta}{\dashrightarrow}X\ov{\overline s}{\Longleftarrow}A]\in\widetilde{\mathbb{E}}(C,A) $. By definition we have
$$ \widetilde{\mathbb{E}}(C,A)=\{[\overline{ t}\backslash \overline{ \delta}/\overline{ s}]|X,Z\in\C,\overline{s},\overline{t}\in\overline{\mathcal{F}},\delta\in\overline{\mathbb{E}}(Z,X) \}.          $$\end{construction}

Next, for any $\alpha,\gamma\in\widetilde\C$,  we will define $\alpha_\ast[\overline{ t}\backslash \overline{ \delta}/\overline{ s}]$ and $\gamma\uas[\overline{ t}\backslash \overline{ \delta}/\overline{ s}]$.
First of all, we have the following observation.

\begin{remark}\label{rem4} By (MR2), any morphism $\alpha \in\widetilde\C(A,B)$ can be written as
$$\alpha =\overline{Q}(\overline{s})^{-1}\circ\overline{Q}(\overline{f})=\overline{Q}(\overline{g})\circ\overline{Q}(\overline{t})^{-1}$$
by some pair of morphism $A\ov{\overline f}{\rightarrow}B'\ov{\overline s}{\Longleftarrow}B$ and $A\ov{\overline t}{\Longleftarrow}A'\ov{\overline g}{\rightarrow}B$.
\end{remark}

\begin{construction}\label{def0}
Let $[\overline{ t}\backslash \overline{ \delta}/\overline{ s}]=[C\ov{\overline t}{\Longleftarrow}Z\ov{\overline \delta}{\dashrightarrow}X\ov{\overline s}{\Longleftarrow}A]\in\widetilde{\mathbb{E}}(C,A) $.

$\bigstar$ For any morphism $\alpha\in\widetilde{\C}(A,A')$, we can write it as $\alpha=\overline{Q}(\overline{u})^{-1}\circ\overline{Q}(\overline{a})$ with some $\overline a\in\overline{\C}(A,D)$ and $\overline u\in\overline{\mathcal{F}}(A',D)$ by Remark \ref{rem4}. By (MR2), there exists a commutative square in $\overline\C$ as follows:
$$\xymatrix{X\ar@{}[dr]|-{\circlearrowleft}\ar[r]^{\overline{a'}}\ar@{<=}[d]_{\overline{s}}&{\clubsuit}\ar@{<=}[d]_{\overline{s'}}\\
A\ar[r]^{\overline{a}}&D.}$$
We take $\alpha_\ast[\overline{ t}\backslash \overline{ \delta}/\overline{ s}]=[\overline{ t}\backslash \overline{ a'_\ast\delta}/\overline{ s'\circ u}]=[C\ov{\overline t}{\Longleftarrow}Z\ov{\overline{ a'_\ast\delta}}{\dashrightarrow}\clubsuit\ov{\overline{ s'\circ u}}{\Longleftarrow}A']$.

$\bigstar$ Dually, for any morphism $\gamma\in\widetilde{\C}(C',C)$, we can write it as $\gamma=\overline{Q}(\overline{c})\circ\overline{Q}(\overline{v})^{-1}$ with some $\overline c\in\overline{\C}(E,C)$ and $\overline v\in\overline{\mathcal{F}}(E,C')$ by Remark \ref{rem4}. By (MR2), there exists a commutative square in $\overline\C$ as follows:
$$\xymatrix{\spadesuit\ar@{}[dr]|-{\circlearrowleft}\ar[r]^{\overline{c'}}\ar@{=>}[d]_{\overline{t'}}&{Z}\ar@{=>}[d]_{\overline{t}}\\
E\ar[r]^{\overline{c}}&C.}$$
We take $\gamma\uas[\overline{ t}\backslash \overline{ \delta}/\overline{ s}]=[{\overline{ v\circ t'}}\backslash \overline{ c'^\uas\delta}/\overline{ s}]=[C'\ov{{\overline{ v\circ t'}}}{\Longleftarrow}\spadesuit\ov{\overline{ c'^\uas\delta}}{\dashrightarrow}X\ov{\overline{ s }}{\Longleftarrow}A]$.

\end{construction}

\begin{lemma}\rm \label{Lem5}
The Construction \ref{def0} is well-defined.
\end{lemma}
\begin{proof} This result was proved in \cite[Lemma 3.24]{HYA} for the case that $\C$ is an extriangulated category. But their proof can be applied to the context of an $n$-exangulated category without any
change.
\end{proof}

\begin{remark}\label{rem5}
Keep the notations in Construction \ref{def0}. For any $[\overline{ t}\backslash \overline{ \delta}/\overline{ s}]\in\widetilde{\mathbb{E}}(C,A) $, $\alpha=\overline{Q}(\overline{u})^{-1}\circ\overline{Q}(\overline{a})\in\widetilde{\C}(A,A')$, $\gamma=\overline{Q}(\overline{c})\circ\overline{Q}(\overline{v})^{-1}\in\widetilde{\C}(C',C)$, the element
$$\gamma\uas\alpha_\ast[\overline{ t}\backslash \overline{ \delta}/\overline{ s}]=[\overline{v\circ t'}\backslash \overline{c'^\uas a'_\ast\delta}/\overline{ s'\circ u}]=[C'\ov{\overline{v\circ t'}}{\Longleftarrow}\spadesuit\ov{\overline{c'^\uas a'_\ast\delta}}{\dashrightarrow}\clubsuit\ov{\overline{ s'\circ u}}{\Longleftarrow}A']$$
is defined by the following diagram
$$\xymatrix{&&\spadesuit\ar@{}[dd]|{\circlearrowleft}\ar[rd]^{\overline{c'}}\ar@{=>}[ld]_{\overline{t'}}\ar@{-->}[rrr]^{\overline{c'^\uas a'_\ast\delta}}&&&\clubsuit\ar@{}[dd]|{\circlearrowleft}\ar@{<=}[rd]^{\overline{s'}}&&\\
&E\ar[rd]_{\overline{c}}\ar@{=>}[ld]_{\overline{v}}&&Z\ar@{=>}[ld]^{\overline{t}}\ar@{-->}[r]^{\overline{\delta}}&X\ar[ru]_{\overline{a'}}\ar@{<=}[rd]_{\overline{s}}&&D\ar@{<=}[rd]^{\overline{u}}&\\C'&&C&&&A\ar[ru]_{\overline{a}}&&A'.}$$

\end{remark}

\begin{lemma}\label{kl}\rm Let $A$ and $C$ be two any objects in $\C$. The following statements hold.

(1) For any pair of elements $[\overline{ t_i}\backslash \overline{ \delta_i}/\overline{ s_i}]\in\widetilde{\mathbb{E}}(C,A) ~(i=1,2)$, by Proposition \ref{Prop1}, we define the addition operation as
$$[\overline{ t_1}\backslash \overline{ \delta_1}/\overline{ s_1}]+[\overline{ t_2}\backslash \overline{ \delta_2}/\overline{ s_2}]=[\overline{ t}\backslash \overline{ \rho_1}+\overline{ \rho_2}/\overline{ s}]$$
by taking a common denominator so that $[\overline{ t_i}\backslash \overline{ \delta_i}/\overline{ s_i}]=[\overline{ t}\backslash \overline{ \rho_i}/\overline{ s}]$ hold for $i=1,2$. Then $\widetilde{\mathbb{E}}(C,A)$ can be viewed as an abelian group under this addition with zero element $[\overline{ \id}\backslash 0/\overline{ \id}]$.

(2) For any $\overline{\delta}\in\overline{\mathbb{E}}(C,A)$, if we define a map
$$\overline{\mu}_{C,A}:\overline{\mathbb{E}}(C,A)\longrightarrow\widetilde{\mathbb{E}}(C,A)~~~\overline{\delta}\mapsto[\overline{ \id}\backslash \overline{\delta}/\overline{ \id}],$$
then it is a monomorphism.
\end{lemma}
\begin{proof}
(1) The addition operation is well-defined by Proposition \ref{Prop1}. Next, we will check that $\widetilde{\mathbb{E}}(C,A)$ is an abelian group.

$\bullet$ (Commutativity) It is obvious.

$\bullet$ (Associativity) For any $[\overline{ t_i}\backslash \overline{ \delta_i}/\overline{ s_i}]\in\widetilde{\mathbb{E}}(C,A) ~(i=1,2,3)$, we take common denominators so that $[\overline{ t_i}\backslash \overline{ \delta_i}/\overline{ s_i}]=[\overline{ t}\backslash \overline{ \rho_i}/\overline{ s}]$ hold for $i=1,2,3$. So we have
$$([\overline{ t_1}\backslash \overline{ \delta_1}/\overline{ s_1}]+[\overline{ t_2}\backslash \overline{ \delta_2}/\overline{ s_2}])+[\overline{ t_3}\backslash \overline{ \delta_3}/\overline{ s_3}]=[\overline{ t}\backslash \overline{ \rho_1}+\overline{ \rho_2}+\overline{ \rho_3}/\overline{ s}]$$
$$=[\overline{ t_1}\backslash \overline{ \delta_1}/\overline{ s_1}]+([\overline{ t_2}\backslash \overline{ \delta_2}/\overline{ s_2}]+[\overline{ t_3}\backslash \overline{ \delta_3}/\overline{ s_3}])$$

$\bullet$ (Zero element) For any $[\overline{ t}\backslash \overline{ \delta}/\overline{ s}]\in\widetilde{\mathbb{E}}(C,A)$, we have
$$[\overline{ t}\backslash \overline{ \delta}/\overline{ s}]+[\overline{ \id}\backslash 0/\overline{ \id}]=[\overline{ t}\backslash \overline{ \delta}/\overline{ s}]+[\overline{ t}\backslash 0/\overline{ s}]=[\overline{ t}\backslash \overline{ \delta}/\overline{ s}].$$

$\bullet$ (Inverse element) For any $[\overline{ t}\backslash \overline{ \delta}/\overline{ s}]\in\widetilde{\mathbb{E}}(C,A)$, it is clear that $[\overline{ t}\backslash \overline{ -\delta}/\overline{ s}]$ is the inverse element of $[\overline{ t}\backslash \overline{ \delta}/\overline{ s}]$.

(2) It follows from the definition and Remark \ref{rem2}.
\end{proof}

\begin{proposition}
\rm The Lemma \ref{kl} give an additive bifunctor $\widetilde{\E}\colon\widetilde{\C}^{\rm op}\times \widetilde{\C}\to{\rm Ab}$ and a natural transformation $\overline\mu :\overline{\mathbb{E}}\Longrightarrow\widetilde{\mathbb{E}}\circ (\overline{Q}^{\rm op}\times \overline{Q})$.
\end{proposition}
\begin{proof} This result was proved in \cite[Lemma 3.27]{HYA} for the case that $\C$ is an extriangulated category. But their proof can be applied to the context of an $n$-exangulated category without any
change.
\end{proof}

\subsection{Construction of $\widetilde{\s}$}
\begin{definition}\label{Def0}
Let $[\overline{ t}\backslash \overline{ \delta}/\overline{ s}]=[C\ov{\overline t}{\Longleftarrow}X_{n+1}\ov{\overline \delta}{\dashrightarrow}X_0\ov{\overline s}{\Longleftarrow}A]\in\widetilde{\mathbb{E}}(C,A) $ be any $\widetilde{\mathbb{E}}$-extension. Take $$\s(\delta)=[X_0\xrightarrow{x_0}X_1\xrightarrow{x_1}\cdots\xrightarrow{x_{n-2}}X_{n-1}
\xrightarrow{x_{n-1}}X_n\xrightarrow{x_n}X_{n+1}],$$
and put
$$\widetilde{\s}([\overline{ t}\backslash \overline{ \delta}/\overline{ s}])=[A\xrightarrow{Q(x_0\circ s)}X_1\xrightarrow{Q(x_1)}\cdots\xrightarrow{Q(x_{n-2})}X_{n-1}
\xrightarrow{Q(x_{n-1})}X_n\xrightarrow{Q (t\circ x_n)}C]$$
in $\widetilde{\C}$.
\end{definition}
The following Lemma show that Definition \ref{Def0} is well-defined.
\begin{lemma}\rm \label{Lem6}
Let $A,C\in\C$ be any pair of objects. Suppose that $[\overline{ t}\backslash \overline{ \delta}/\overline{ s}]=(C\ov{\overline t}{\Longleftarrow}X_{n+1}\ov{\overline \delta}{\dashrightarrow}X_0\ov{\overline s}{\Longleftarrow}A)$ and $[\overline{ t'}\backslash \overline{ \delta'}/\overline{ s'}]=(C\ov{\overline t'}{\Longleftarrow}X'_{n+1}\ov{\overline \delta'}{\dashrightarrow}X'_0\ov{\overline s'}{\Longleftarrow}A)$ be two triplets satisfying $t,t',s,s'\in\overline{\mathcal{F}}, \delta\in\overline{\mathbb{E}}(X_{n+1},X_{0}), \delta'\in\overline{\mathbb{E}}(X'_{n+1},X'_{0})$. Let
$$\s(\delta)=[X_0\xrightarrow{x_0}X_1\xrightarrow{x_1}\cdots\xrightarrow{x_{n-2}}X_{n-1}
\xrightarrow{x_{n-1}}X_n\xrightarrow{x_n}X_{n+1}]$$ and
$$\s(\delta')=[X'_0\xrightarrow{x'_0}X'_1\xrightarrow{x'_1}\cdots\xrightarrow{x'_{n-2}}X'_{n-1}
\xrightarrow{x'_{n-1}}X'_n\xrightarrow{x'_n}X'_{n+1}].$$

If $[\overline{ t}\backslash \overline{ \delta}/\overline{ s}]=[\overline{ t'}\backslash \overline{ \delta'}/\overline{ s'}]$ in $\widetilde{\mathbb{E}}(C,A)$, then
$$[A\xrightarrow{Q(x_0\circ s)}X_1\xrightarrow{Q(x_1)}\cdots\xrightarrow{Q(x_{n-2})}X_{n-1}
\xrightarrow{Q(x_{n-1})}X_n\xrightarrow{Q (t\circ x_n)}C]=$$
$$[A\xrightarrow{Q(x'_0\circ s')}X'_1\xrightarrow{Q(x'_1)}\cdots\xrightarrow{Q(x'_{n-2})}X'_{n-1}
\xrightarrow{Q(x'_{n-1})}X'_n\xrightarrow{Q (t'\circ x'_n)}C]$$ holds as sequences in $\widetilde{\C}$.
\end{lemma}
\begin{proof} By the definition of the equivalence relation, we have the following commutative diagram
$$\xymatrix{&{X_{n+1}}\ar@{=>}[ld]_{\overline{t}}\ar@{}[ld]^{\circlearrowleft}\ar@{-->}[r]^{\overline{\delta}}\ar@{<=}[d]^{\overline{v}}&{X_{0}}\ar@{}[rd]_{\circlearrowleft}\ar@{<=}[rd]^{\overline{s}}\ar@{=>}[d]_{\overline{u}}&\\
C\ar@{<=}[r]_{\overline{t'}}&X'_{n+1}\ar@{-->}[r]_{\overline{\delta'}}&X'_0\ar@{<=}[r]_{\overline{s'}} &A.}$$
That is to say , there exist $\overline{u}\in\overline{\mathcal{F}}(X_0,X'_0)$ and $\overline{v}\in\overline{\mathcal{F}}(X'_{n+1},X'_{n+1})$ such that $\overline{s'}=\overline{u}\circ\overline{s},~\overline{t'}=\overline{t}\circ\overline{v}$, and
$\overline {u\uas v_\ast \delta}=\overline {\delta'}$. By the definition of $\overline{\mathbb{E}}$, there exists $u'\in\overline{\mathcal{F}}(X'_0,X''_0)$ such that $u'^\uas(u\uas v_\ast \delta)= u'^\uas{\delta'}$.

So replacing $u,s', \delta'$ by $u'u,u's', u'_\ast\delta'$ respectively, we may assume that
${u\uas v_\ast \delta}={\delta'}$ holds from the beginning. Under this circumstance, for
$$\s(u_\ast\delta)=[X'_0\xrightarrow{y_0}Y_1\xrightarrow{y_1}\cdots\xrightarrow{y_{n-2}}Y_{n-1}
\xrightarrow{y_{n-1}}Y_n\xrightarrow{y_n}X_{n+1}],$$
there exist $\mathbf{w_1,\cdots,w_n,w'_1,\cdots,w'_n} \in \overline{\mathcal{F}}$ such that the following diagram is commutative in $\overline \C$ by (MR3):

$$\xymatrix{
X_0\ar[r]^{\overline {x_0}}\ar@{}[dr] \ar@{=>}[d]^{\overline u} &X_1 \ar[r]^{\overline {x_1}} \ar@{}[dr]\ar@{=>}[d]^{ {\mathbf{w_1}}} &X_2 \ar[r]^{\overline {x_2}} \ar@{}[dr]\ar@{=>}[d]^{ {\mathbf{w_2}}}&\cdot\cdot\cdot \ar[r]^{\overline {x_{n-2}}}\ar@{}[dr] &X_{n-1} \ar[r]^{\overline {x_{n-1}}} \ar@{}[dr]\ar@{=>}[d]^{ {\mathbf{w_{n-1}}}}&{X_{n}}\ar[r]^{\overline {x_{n}}}\ar@{}[dr]\ar@{=>}[d]^{ {\mathbf{w_n}}}&{X_{n+1}} \ar@{}[dr]\ar@{=}[d]^{} &\\X'_0\ar[r]^{\overline {y_{0}}}\ar@{}[dr] \ar@{=}[d]^{} &Y_1 \ar[r]^{\overline {y_{1}}} \ar@{}[dr]\ar@{<=}[d]^{ {\mathbf{w'_{1}}}}&Y_2 \ar[r]^{\overline {y_{2}}} \ar@{}[dr]\ar@{<=}[d]^{ {\mathbf{w'_{2}}}}&\cdot\cdot\cdot \ar[r]^{\overline {y_{n-2}}}\ar@{}[dr] &Y_{n-1} \ar[r]^{\overline {y_{n-1}}} \ar@{}[dr]\ar@{<=}[d]^{ {\mathbf{w'_{n-1}}}}&Y_n\ar[r]^{\overline {y_{n}}} \ar@{}[dr]\ar@{<=}[d]^{ {\mathbf{w'_{n}}}}&{X_{n+1}} \ar@{}[dr]\ar@{<=}[d]^{\overline v} &\\
X'_0 \ar[r]_{\overline {x'_{0}}}&X'_1 \ar[r]_{\overline {x'_{1}}} & X'_2 \ar[r]_{\overline {x'_{2}}}  & \cdots \ar[r]_{{\overline {x'_{n-2}}}} & {X'_{n-1}}\ar[r]_{\overline {x'_{n-1}}}  &{X'_{n}}\ar[r]_{\overline {x'_{n}}}  &{X'_{n+1}}  &}
$$
Take $\alpha_i=\overline {Q}\mathbf{(w'_{i})^{-1}}\circ \overline {Q}\mathbf{(w_{i})}\in\textrm{Iso}(\widetilde{\C})$, where $i=1,2,\cdots,n$. We have the following commutative diagram
$$\xymatrix@C=1.2cm{
A\ar[r]^{Q(x_0\circ s)}\ar@{=}[d]^{}& X_1 \ar[r]^{Q(x_1)}\ar[d]_{\alpha_1}^{\cong} & X_2 \ar[r]^{Q(x_2)}\ar[d]_{\alpha_2}^{\cong}& \cdots \ar[r]^{Q(x_{n-2})}&X_{n-1}\ar[d]_{\alpha_{n-1}}^{\cong}\ar[r]^{Q(x_{n-1})}&X_{n}\ar[r]^{Q(t\circ x_{n})}\ar[d]_{\alpha_{n}}^{\cong}&C\ar@{}[dr]\ar@{=}[d]^{} &\\
A\ar[r]_{Q(x'_0\circ s')}&X'_1 \ar[r]_{Q(x'_1)} & X'_2 \ar[r]_{Q(x'_2)}  & \cdots \ar[r]_{Q(x'_{n-2})} &X'_{n-1}\ar[r]_{Q(x'_{n-1})}& X'_{n} \ar[r]_{Q(t'\circ x'_{n})}&C &}$$
in $\widetilde{\C}$, this is exactly what we want.
\end{proof}
\begin{proposition}\rm \label{pro6}
Let $[\overline{ t}\backslash \overline{ \delta}/\overline{ s}]=[C\ov{\overline t}{\Longleftarrow}X_{n+1}\ov{\overline \delta}{\dashrightarrow}X_0\ov{\overline s}{\Longleftarrow}A]\in\widetilde{\mathbb{E}}(C,A) $ be any $\widetilde{\mathbb{E}}$-extension.
Take $$\widetilde{\s}([\overline{ t}\backslash \overline{ \delta}/\overline{ s}])=[A\xrightarrow{Q(x_0\circ s)}X_1\xrightarrow{Q(x_1)}\cdots\xrightarrow{Q(x_{n-2})}X_{n-1}
\xrightarrow{Q(x_{n-1})}X_n\xrightarrow{Q (t\circ x_n)}C],$$
where $$\s(\delta)=[X_0\xrightarrow{x_0}X_1\xrightarrow{x_1}\cdots\xrightarrow{x_{n-2}}X_{n-1}
\xrightarrow{x_{n-1}}X_n\xrightarrow{x_n}X_{n+1}].$$ If any distinguished $n$-exangle in $\C$ induces a weak kernel-cokernel sequence in $\widetilde{\mathcal{C}}$, then the following statements hold.

(1) The sequence is exact
$$
\widetilde{\C}(-,A)\xrightarrow{\widetilde{\C}(-,\  Q(x_0\circ s))}\widetilde{\C}(-,X_1)\xrightarrow{\widetilde{\C}(-,\ Q(x_1))}\cdots\xrightarrow{\widetilde{\C}(-,\ Q(t\circ x_n))}\widetilde{\C}(-,C)\xrightarrow{{[\overline{ t}\backslash \overline{ \delta}/\overline{ s}]\ssh}}\widetilde{\mathbb{E}}(-,A).
$$

(2) The sequence is exact
$$
\widetilde{\C}(C,-)\xrightarrow{\widetilde{\C}(Q(t\circ x_n),\ -)}\widetilde{\C}(X_{n},-)\xrightarrow{\widetilde{\C}(Q({x_{n-1}}),\ -)}\cdots\xrightarrow{\widetilde{\C}(Q(x_0\circ s),\ -)}\widetilde{\C}(A,-)\xrightarrow{{[\overline{ t}\backslash \overline{ \delta}/\overline{ s}]\ush}}\widetilde{\mathbb{E}}(C,-).
$$
\end{proposition}

\begin{proof}
We only prove that $(1)$, dually one can prove $(2)$.

Consider the commutativity of the following diagram
$$\xymatrix@C=2.6cm{
\widetilde{\C}(-,A)\ar[r]^{\widetilde{\C}(-,\  Q(x_0\circ s))}\ar[d]^{\widetilde{\C}(-,\  Q(s))}_\cong& \widetilde{\C}(-,X_1) \ar[r]^{\widetilde{\C}(-,\  Q(x_1))}\ar@{=}[d]^{} & \widetilde{\C}(-,X_2) \ar[r]^{\widetilde{\C}(-,\  Q(x_2))}\ar@{=}[d]^{}& \cdots  &\\
\widetilde{\C}(-,X_0)\ar[r]_{\widetilde{\C}(-,\  Q(x_0))}&\widetilde{\C}(-,X_1)\ar[r]_{\widetilde{\C}(-,\  Q(x_1))} & \widetilde{\C}(-,X_2)\ar[r]_{\widetilde{\C}(-,\  Q(x_2))}  & \cdots &}$$
$$\xymatrix@C=2.6cm{\cdots\ar[r]^{\widetilde{\C}(-,\  Q(x_{n-1}))}&\widetilde{\C}(-,X_{n}) \ar@{=}[d]^{}\ar[r]^{\widetilde{\C}(-,\ Q(t\circ x_{n}))}&\widetilde{\C}(-,C) \ar[r]^{{[\overline{ t}\backslash \overline{ \delta}/\overline{ s}]\ssh}}&\widetilde{\mathbb{E}}(-,A)\ar@{}[dr]\ar[d]^{Q(s)_\ast}_\cong &\\
\cdots\ar[r]_{\widetilde{\C}(-,\  Q(x_{n-1}))} &\widetilde{\C}(-,X_n)\ar[r]_{\widetilde{\C}(-,\ Q(x_{n}))}& \widetilde{\C}(-,X_{n+1})\ar[u]_{\widetilde{\C}(-,\ Q(t))}^{\cong}\ar[r]_{{[\overline{ \id}\backslash \overline{ \delta}/\overline{ \id}]\ssh}}&\widetilde{\mathbb{E}}(-,X_0) &}
$$
In order to prove the exactness of the top row, it suffices to show the exactness of the bottom row. Since any distinguished $n$-exangle in $\C$ induces a weak kernel-cokernel sequence in $\widetilde{\mathcal{C}}$ by assumption, the sequence
$$
\widetilde{\C}(-,X_0)\xrightarrow{\widetilde{\C}(-,\  Q(x_0))}\widetilde{\C}(-,X_1)\xrightarrow{\widetilde{\C}(-,\ Q(x_1))}\cdots\xrightarrow{\widetilde{\C}(-,\ Q(x_{n-1}))}\widetilde{\C}(-,X_n)\xrightarrow{\widetilde{\C}(-,\ Q(x_n))}\widetilde{\C}(-,X_{n+1})
$$
is exact. For any $P\in \widetilde{\C}$, we only need to show the exactness of
$$
\widetilde{\C}(P,X_n)\xrightarrow{\widetilde{\C}(P,\ Q(x_n))}\widetilde{\C}(P,X_{n+1})\xrightarrow{{[\overline{ \id}\backslash \overline{ \delta}/\overline{ \id}]\ssh}}\widetilde{\mathbb{E}}(P,X_{0}).
$$

It is clear that ${\rm Im}~\widetilde{\C}(P,\ Q(x_n))\subseteq{\rm Ker}~{[\overline{ \id}\backslash \overline{ \delta}/\overline{ \id}]\ssh}$. We show ${\rm Ker}~{[\overline{ \id}\backslash \overline{ \delta}/\overline{ \id}]\ssh}\subseteq{\rm Im}~\widetilde{\C}(P,\ Q(x_n))$. Suppose that $\beta\in\widetilde{\C}(P,X_{n+1})$ satisfies $[\overline{ \id}\backslash \overline{ \delta}/\overline{ \id}]\ssh(\beta)=\beta\uas[\overline{ \id}\backslash \overline{ \delta}/\overline{ \id}]=0.$ By Remark \ref{rem4}, we can write $\beta=\overline{Q}(\overline{f})\circ\overline{Q}(\overline{g})^{-1}$ by some pair of morphism
$P\ov{\overline g}{\Longleftarrow}P'\ov{\overline f}{\rightarrow}X_{n+1}$. Then $\beta\uas[\overline{ \id}\backslash \overline{ \delta}/\overline{ \id}]=[\overline{ g}\backslash \overline{ f\uas\delta}/\overline{ \id}]$ by definition. It follows $\overline{f\uas\delta}=0$, hence ${m\uas(f\uas\delta)}=0$ for some $m\in\overline{\mathcal{F}}(P'',P')$ by Remark \ref{rem2}. Note that the following exact sequence
$${\C}(P'',X_n)\xrightarrow{{\C}(P'',x_n)}{\C}(P'',X_{n+1})\xrightarrow{ \delta\ssh}{\mathbb{E}}(P'',X_{0}),
$$
there exists $h\in \C(P'',X_{n})$ such that $x_n\circ h=f\circ m$. Set $\alpha={\overline Q}{ (\overline h)}{\overline Q}(\overline {gm})^{-1}\in\widetilde{\C}(P,X_{n}) $, so we have $Q(x_{n})\alpha=\beta$.
\end{proof}

{\bf Now we are ready to prove Theorem \ref{th}.}

\begin{proof}
(1) { {\bf Necessity.~} It is trivial.}

{ {\bf Sufficiency.~}} (C1) follows from Proposition \ref {pro6}. (C2) and (C2') are obvious from the definition. We only need to prove (C3), dually, one can prove (C3').

Let $[\overline{ t}\backslash \overline{ \delta}/\overline{ s}]=[C\ov{\overline t}{\Longleftarrow}X_{n+1}\ov{\overline \delta}{\dashrightarrow}X_0\ov{\overline s}{\Longleftarrow}A]\in\widetilde{\mathbb{E}}(C,A) $ be any $\widetilde{\mathbb{E}}$-extension and $\alpha\in\widetilde{\C}(A,A')$ be any morphism. By Remark \ref{rem4}, we can write $\alpha=\overline{Q}(\overline{u})^{-1}\circ\overline{Q}(\overline{a})$ by some pair of morphism
$A\ov{\overline a}{\rightarrow}D\ov{\overline u}{\Longleftarrow}A'$. Then by construction $\alpha_\ast[\overline{ t}\backslash \overline{ \delta}/\overline{ s}]=[\overline{ t}\backslash \overline{ a'_\ast\delta}/\overline{ s'\circ u}]$ is given by using a commutative square in $\overline\C$
$$\xymatrix{A\ar@{}[dr]|-{\circlearrowleft}\ar[r]^{\overline{a}}\ar@{=>}[d]_{\overline{s}}&{D}\ar@{=>}[d]_{\overline{s'}}\\
X_0\ar[r]^{\overline{a'}}&\clubsuit.}$$
Applying (C3) for $(\C,\E,\s)$, there exists the following commutative diagram of $\s$-distinguished $n$-exangles
\begin{equation}\xymatrix{
X_0 \ar[r]^{x_0}\ar[d]^{a'} & X_1 \ar[r]^{x_1}\ar[d]^{b_1} & X_2 \ar[r]^{x_2} \ar[d]^{b_2}& \cdots \ar[r]^{x_{n-1}} & X_{n} \ar[r]^{x_{n}}\ar[d]^{b_n}&X_{n+1}\ar@{=}[d]\ar@{-->}[r]^-{\delta} & \\
\clubsuit \ar[r]^{y_0} & Y_1 \ar[r]^{y_1} & Y_2\ar[r]^{y_2} & \cdots \ar[r]^{y_{n-1}} & Y_n \ar[r]^{y_{n}} & X_{n+1}\ar@{-->}[r]^-{a'_{\ast}\del} & }
\end{equation}
such that
\begin{equation}X_0\xrightarrow{\left[
                                            \begin{smallmatrix}
                                              -x_0 \\
                                              a' \\
                                            \end{smallmatrix}
                                          \right]}
 X_1\oplus \clubsuit\xrightarrow{\left[
                                            \begin{smallmatrix}
                                              -x_{1} & 0 \\
                                              b_{1} & y_0
                                            \end{smallmatrix}
                                          \right]}
X_2\oplus Y_1\xrightarrow{\left[
                                            \begin{smallmatrix}
                                             -x_{2} & 0 \\
                                             b_{2} & y_1
                                            \end{smallmatrix}
                                          \right]}
\cdots\xrightarrow{\left[
                                            \begin{smallmatrix}
                                              -x_{n-1} & 0 \\
                                              b_{n-1} & y_{n-2} \\
                                            \end{smallmatrix}
                                          \right]\ }
X_{n}\oplus Y_{n-1}\xrightarrow{\left[\begin{smallmatrix}
                                            b_{n}, & y_{n-1} \\
                                            \end{smallmatrix}
                                          \right]}
Y_{n}\overset{(y_n){^{\ast}}\delta}{\dashrightarrow}\end{equation}
is an $\s$-distinguished $n$-exangle. By definition we have
$$\widetilde{\s}([\overline{ t}\backslash \overline{ \delta}/\overline{ s}])=[A\xrightarrow{Q(x_0\circ s)}X_1\xrightarrow{Q(x_1)}\cdots\xrightarrow{Q(x_{n-2})}X_{n-1}
\xrightarrow{Q(x_{n-1})}X_n\xrightarrow{Q (t\circ x_n)}C]$$
and
$$\widetilde{\s}(\alpha_\ast[\overline{ t}\backslash \overline{ \delta}/\overline{ s}])=[A'\xrightarrow{Q(y_0\circ s'\circ u)}Y_1\xrightarrow{Q(y_1)}\cdots\xrightarrow{Q(y_{n-2})}Y_{n-1}
\xrightarrow{Q(y_{n-1})}Y_n\xrightarrow{Q (t\circ y_n)}C].$$
Take $x''=y_0\circ s'\circ u$. The commutative diagram of (3.3) induces the following commutative diagram of $\widetilde\s$-distinguished $n$-exangles
$$\xymatrix@C=1.2cm{
A \ar[r]^{Q(x_0\circ s)}\ar[d]^{\alpha} & X_1 \ar[r]^{Q(x_1)}\ar[d]^{Q(b_1)} & X_2 \ar[r]^{Q(x_2)} \ar[d]^{Q(b_2)}& \cdots \ar[r]^{Q(x_{n-2})}& X_{n-1} \ar[r]^{Q(x_{n-1})}\ar[d]^{Q(b_{n-1})} & X_{n} \ar[r]^{Q(t\circ x_{n})}\ar[d]^{Q(b_n)}&C\ar@{=}[d]\ar@{-->}[r]^-{[\overline{ t}\backslash \overline{ \delta}/\overline{ s}]} & \\
A' \ar[r]_{Q(x'')} & Y_1 \ar[r]_{Q(y_1)} & Y_2\ar[r]_{Q(y_2)} & \cdots \ar[r]_{Q(y_{n-2})} & Y_{n-1} \ar[r]_{Q(y_{n-1})}& Y_n \ar[r]_{Q(t\circ y_{n})} & C\ar@{-->}[r]^-{\alpha_\ast[\overline{ t}\backslash \overline{ \delta}/\overline{ s}]} &. }
$$
Since ${Q(t\circ y_{n})}^\ast[\overline{ t}\backslash \overline{ \delta}/\overline{ s}]=[\overline{ \id_{Y_{n}}}\backslash \overline{y_{n}^\ast \delta}/\overline{ s}]$, we need to show that

$$A\xrightarrow{\left[
                                            \begin{smallmatrix}
                                              -Q(x_0\circ s) \\
                                              \alpha \\
                                            \end{smallmatrix}
                                          \right]}
 X_1\oplus A'\xrightarrow{\left[
                                            \begin{smallmatrix}
                                              -Q(x_{1}) & 0 \\
                                              Q(b_{1}) & Q(x'')
                                            \end{smallmatrix}
                                          \right]}
X_2\oplus Y_1\xrightarrow{\left[
                                            \begin{smallmatrix}
                                             -Q(x_{2}) & 0 \\
                                             Q(b_{2}) & Q(y_1)
                                            \end{smallmatrix}
                                          \right]}
\cdots$$
\begin{equation}\cdots\xrightarrow{\left[
                                            \begin{smallmatrix}
                                              -Q(x_{n-1}) & 0 \\
                                             Q( b_{n-1}) & Q(y_{n-2}) \\
                                            \end{smallmatrix}
                                          \right]\ }
X_{n}\oplus Y_{n-1}\xrightarrow{\left[\begin{smallmatrix}
                                            Q(b_{n}), & Q(y_{n-1}) \\
                                            \end{smallmatrix}
                                          \right]}
Y_{n}\mathop{----\to}\limits^{[\overline{ \id_{Y_{n}}}\backslash \overline{y_{n}^\ast \delta}/\overline{ s}]}\end{equation}
is an $\widetilde\s$-distinguished $n$-exangle.

Note that (3.4) is an $\s$-distinguished $n$-exangle, we have

$$ \widetilde\s([\overline{ \id_{Y_{n}}}\backslash \overline{y_{n}^\ast \delta}/\overline{ s}])=[A\xrightarrow{Q(\left[
                                            \begin{smallmatrix}
                                              -x_0 \\
                                              a' \\
                                            \end{smallmatrix}
                                          \right]\circ s)}
 X_1\oplus \clubsuit\xrightarrow{Q(\left[
                                            \begin{smallmatrix}
                                              -x_{1} & 0 \\
                                              b_{1} & y_{0}
                                            \end{smallmatrix}
                                          \right])}
X_2\oplus Y_1\xrightarrow{Q(\left[
                                            \begin{smallmatrix}
                                             -x_{2} & 0 \\
                                             b_{2} & y_1
                                            \end{smallmatrix}
                                          \right])}
\cdots$$
$$\xrightarrow{Q(\left[
                                            \begin{smallmatrix}
                                              -x_{n-1} & 0 \\
                                             b_{n-1} & y_{n-2} \\
                                            \end{smallmatrix}
                                          \right]) }
X_{n}\oplus Y_{n-1}\xrightarrow{Q(\left[\begin{smallmatrix}
                                            b_{n}, & y_{n-1} \\
                                            \end{smallmatrix}
                                          \right])}
Y_{n}]$$
by definition. Hence (3.5) indeed becomes an $\widetilde\s$-distinguished $n$-exangle. In fact, we have the following commutative diagram in $\widetilde\C$
$$\xymatrix@C=2.6cm{
A\ar[r]^{\left[
                                            \begin{smallmatrix}
                                              -Q(x_0\circ s) \\
                                              \alpha \\
                                            \end{smallmatrix}
                                          \right]}\ar@{=}[d]^{}& X_1\oplus A' \ar[r]^{\left[
                                            \begin{smallmatrix}
                                              -Q(x_{1}) & 0 \\
                                              Q(b_{1}) & Q(x'')
                                            \end{smallmatrix}
                                          \right]}\ar[d]^{\beta}_\cong & X_2\oplus Y_1  \ar[r]^{\left[
                                            \begin{smallmatrix}
                                             -Q(x_{2}) & 0 \\
                                             Q(b_{2}) & Q(y_1)
                                            \end{smallmatrix}
                                          \right]}\ar@{=}[d]^{}& \cdots  &\\
A\ar[r]_{Q(\left[
                                            \begin{smallmatrix}
                                              -x_0 \\
                                              a' \\
                                            \end{smallmatrix}
                                          \right]\circ s)}&X_1\oplus \clubsuit\ar[r]_{Q(\left[
                                            \begin{smallmatrix}
                                              -x_{1} & 0 \\
                                              b_{1} & y_{0}
                                            \end{smallmatrix}
                                          \right])} & X_2\oplus Y_1 \ar[r]_{Q(\left[
                                            \begin{smallmatrix}
                                             -x_{2} & 0 \\
                                             b_{2} & y_1
                                            \end{smallmatrix}
                                          \right])}  & \cdots &}$$
$$\xymatrix@C=3.5cm{\cdots \ar[r]^{\left[
                                            \begin{smallmatrix}
                                              -Q(x_{n-2}) & 0 \\
                                             Q( b_{n-2}) & Q(y_{n-3}) \\
                                            \end{smallmatrix}
                                          \right]}&X_{n-1}\oplus Y_{n-2} \ar@{=}[d]^{}\ar[r]^{\left[
                                            \begin{smallmatrix}
                                              -Q(x_{n-1}) & 0 \\
                                             Q( b_{n-1}) & Q(y_{n-2}) \\
                                            \end{smallmatrix}
                                          \right]}&X_{n}\oplus Y_{n-1} \ar@{=}[d]^{} \ar[r]^{\left[\begin{smallmatrix}
                                            Q(b_{n}), & Q(y_{n-1}) \\
                                            \end{smallmatrix}
                                          \right]}&Y_{n}\ar@{=}[d]^{}&\\\cdots
\ar[r]_{Q(\left[
                                            \begin{smallmatrix}
                                              -x_{n-2} & 0 \\
                                             b_{n-2} & y_{n-3} \\
                                            \end{smallmatrix}
                                          \right])} & X_{n-1}\oplus Y_{n-2} \ar[r]_{Q(\left[
                                            \begin{smallmatrix}
                                              -x_{n-1} & 0 \\
                                             b_{n-1} & y_{n-2} \\
                                            \end{smallmatrix}
                                          \right])}& X_{n}\oplus Y_{n-1}\ar[r]_{Q(\left[\begin{smallmatrix}
                                            b_{n}, & y_{n-1} \\
                                            \end{smallmatrix}
                                          \right])}&Y_{n}, &}
$$
where $\beta=\id_{X_1}\oplus Q(s'\circ u)\in\widetilde\C( X_1\oplus A', X_1\oplus \clubsuit)$ is an isomorphism.

(2) It is clear.

\end{proof}

{ \begin{remark}
(1) Let $(\C, \mathbb{E}, \mathfrak{s})$ be an extriangulated category and $\mathcal{F}\subseteq \mathcal{M}_{\C}$ satisfies {\rm(M0)}.
Suppose that $\overline{\mathcal{F}}$ satisfies {\rm(MR1)}, {\rm(MR2)} and {\rm(MR3)}, then
any $\E$-triangle in $\C$ induces a weak kernel-cokernel sequence in $\widetilde{\C}$. Indeed,
assume that $A\xrightarrow{~\alpha_0~}B\xrightarrow{~\alpha_1~}C\overset{\delta}{\dashrightarrow}$
is an $\E$-triangle in $\C$. We know that
$$
\widetilde{\C}(-,A)\xrightarrow{\widetilde{\C}(-\, Q({\alpha_0}))}\widetilde{\C}(-,B)\xrightarrow{\widetilde{\C}(-\,Q({\alpha_1}))}\widetilde{\C}(-,C)
$$
and
$$
\widetilde{\C}(C,-)\xrightarrow{\widetilde{\C}(Q({\alpha_1}),\ -)}\widetilde{\C}(B,-)\xrightarrow{\widetilde{\C}(Q({\alpha_0}),\ -)}\widetilde{\C}(A,-)
$$
are exact by the proof of \cite[Proposition 3.31]{HYA}. By Theorem \ref{th}, we have that the localization of $\C$ by $\mathcal{F}$ gives a weakly extriangulated category $(\widetilde{\C}, \widetilde{\mathbb{E}}, \widetilde{\mathfrak{s}})$.

This result is just the Theorem 3.5 (1) in \cite{HYA}. Hence our result recovers and extends a result of Nakaoka-Ogawa-Sakai \cite[Theorem 3.5 (1)]{HYA}.

(2) When can the weakly extriangulated category $\widetilde{\C}$ become an extriangulated category, Nakaoka-Ogawa-Sakai  gave a sufficient condition, i.e. if $\mathcal{F}$ moreover satisfies {\rm(MR4)}, then $(\widetilde{\C}, \widetilde{\mathbb{E}}, \widetilde{\mathfrak{s}})$ is an extriangulated category, where
{\rm(MR4)} is as follows.

(MR4) $\mathbf{\overline{\mathcal{M}}_\textrm{inf}}:= \{\mathbf{v} \circ \overline{x} \circ \mathbf{u} ~| ~x ~\textrm{is ~an}~ \mathfrak{s}\textrm {-inflation}, ~\mathbf{u}, ~\mathbf{v} \in \overline{\mathcal{F}}\}\subseteq \overline{\mathcal{M}}_{\C}$ is closed by compositions. Dually, $\mathbf{\overline{\mathcal{M}}_\textrm{def}}:= \{\mathbf{v} \circ \overline{y} \circ \mathbf{u}~ |~ y~ \textrm{is~an}~ \mathfrak{s}\textrm {-deflation}, ~\mathbf{u}, ~\mathbf{v} ~\in  \overline{\mathcal{F}}\}\subseteq \overline{\mathcal{M}}_{\C}$ is closed
by compositions.

Let $\C$ be a weakly extriangulated category and $(a,b,c)$ be any morphism of $\mathbb{E}$-triangles.  Note that, in \cite{HYA}, one of the key arguments in the proof is that if $a,c$ are isomorphisms, then so is $b$, (for more details, see \cite[Lemma 3.32 and Theorem 3.5 (3)]{HYA} ). While in our general context we do not have this fact and therefore we just assume that $\widetilde{\C}$ satisfies {\rm(C4)}, moreover we also have a class of examples to explain our result, see Example \ref {example:3.0}.
\end{remark}}

\begin{example}\label{example:3.0}
$n$-exangulated categories obtained by ideal quotients can be viewed as a particular type of the localization under the certain condition. More precisely, let $(\C, \mathbb{E}, \mathfrak{s})$ be an $n$-exangulated category and $\I\subseteq{\C}$ be any full additive subcategory closed by isomorphisms and direct summands, whose objects are both projective and injective. Let $p:\C\rightarrow{\C}/[\I]$ denote the ideal quotient. If we take $\mathcal{F}=p^{-1}(\textrm{Iso}~({\C}/[\I]))$, then it is easy to see that the ideal $[\mathcal{N}_{\mathcal{F}}]$ coincides with $[\I]$. As shown in \cite[Theorem 3.1]{HZZ2}, if any distinguished $n$-exangle in $\C$ induces a weak kernel-cokernel sequence in $\overline{\mathcal{C}}$, then $\overline{\mathcal{C}}$ has a natural $n$-exangulated structure $(\overline{\mathcal{C}},\overline{\mathbb{E}},\overline{\mathfrak{s}})$ given by
\begin{itemize}
\item $\overline{\E}(C,A)=\E(C,A)$ for any $A,C\in\C$,

\item For any $\overline{\E}$-extension $\delta\in\overline{\E}(C,A)={ \E}(C,A)$, define
$$\overline{\s}(\delta)=\overline{\s(\delta)}=[A\xrightarrow{~\overline{\alpha_0}~}
B^1\xrightarrow{~\overline{\alpha_1}~}
B^2\xrightarrow{~\overline{\alpha_2}~}\cdots\xrightarrow{~\overline{\alpha_{n-1}}~}B^n\xrightarrow{~\overline{\alpha_{n}}~}C]$$
using $\s(\delta)=[A\xrightarrow{~\alpha_0~}
B^1\xrightarrow{~\alpha_1~}
B^2\xrightarrow{~\alpha_2~}\cdots\xrightarrow{~\alpha_{n-1}~}B^n\xrightarrow{\alpha_{n}}C]$.
\smallskip
\end{itemize}
It is obvious that $\phi=\{\phi_{C,A}=\id: \mathbb{E}(C,A)\longrightarrow\overline{\mathbb{E}}(C,A)   \}_{C,A\in\C}$ gives an exact functor $(p,\phi):(\C,\E,\s)\longrightarrow(\overline\C,\overline\E,\overline\s)$.

As localization by ${\mathcal{F}}$ does not change $(\overline{\mathcal{C}},\overline{\mathbb{E}},\overline{\mathfrak{s}})$ essentially, and there is an equivalence of $n$-exangulated categories $(\overline{\mathcal{C}},\overline{\mathbb{E}},\overline{\mathfrak{s}})\xrightarrow{\simeq}(\widetilde{\mathcal{C}},\widetilde{\mathbb{E}},\widetilde{\mathfrak{s}})$ in the sense of Proposition \ref{rrr}.
\end{example}

{ We revisit Example 3.3 presented in \cite{HZZ2}, which shows that the localization $\widetilde{\mathcal{C}}$ in Theorem \ref{th} is not an $n$-exangulated category in general, even in the special case of ideal quotients.  }

\begin{example}\label{example:3.4}
Let $\Lambda$ be the path algebra of the
quiver
$$1\xrightarrow{~\alpha~}2\xrightarrow{~\beta~}3\xrightarrow{~\gamma~}4$$
with relation $\alpha\beta\gamma=0$.
Then $\mod\Lambda$ has a unique $2$-cluster tilting subcategory $\C$ consisting of all direct
sums of projective modules and injective modules. By \cite[Theorem 3.16]{Ja}, we know that $\C$ is a $2$-abelian category which can be viewed as
a $2$-exangulated category.
The Auslander-Reiten quiver of $\mod\Lambda$ is the following
$$
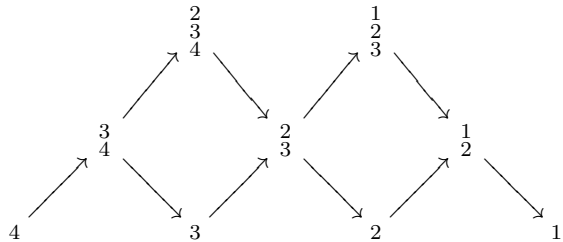
\xymatrix@C=0.7cm@R0.7cm{
&&&{\begin{smallmatrix}
2\\  3\\4
\end{smallmatrix}} \ar[dr]
&&{\begin{smallmatrix}
1\\2\\3
\end{smallmatrix}} \ar[dr]
&& \\
&&{\begin{smallmatrix}
3\\4
\end{smallmatrix}}\ar[dr] \ar[ur]
&& {\begin{smallmatrix}
2\\3
\end{smallmatrix}} \ar[dr] \ar[ur]
&& {\begin{smallmatrix}
1\\2
\end{smallmatrix}} \ar[dr] \\
&{\begin{smallmatrix}
4
\end{smallmatrix}}\ar[ur]
&&{\begin{smallmatrix}
3
\end{smallmatrix}} \ar[ur]
&&  {\begin{smallmatrix}
2
\end{smallmatrix}} \ar[ur]
&& {\begin{smallmatrix}
1
\end{smallmatrix}}}$$
That is $\C:=\add\{\begin{smallmatrix}
4
\end{smallmatrix}, \begin{smallmatrix}
3\\ 4
\end{smallmatrix},\begin{smallmatrix}
2\\ 3\\ 4
\end{smallmatrix},\begin{smallmatrix}
1\\ 2 \\3
\end{smallmatrix},\begin{smallmatrix}
1\\ 2
\end{smallmatrix},\begin{smallmatrix}
1
\end{smallmatrix}\}$. Take
$$\mathcal X:=\add\{\begin{smallmatrix}
2\\ 3\\ 4
\end{smallmatrix}\}\subseteq \add\{ \begin{smallmatrix}
2\\ 3 \\4
\end{smallmatrix}, \begin{smallmatrix}
1\\ 2 \\ 3
\end{smallmatrix}\}=\P\cap\I.$$
Note that $\begin{smallmatrix}
4
\end{smallmatrix}\to \begin{smallmatrix}
2\\ 3\\ 4
\end{smallmatrix}\to \begin{smallmatrix}
1\\ 2
\end{smallmatrix}\to \begin{smallmatrix}
1
\end{smallmatrix}$
is a $2$-exact sequence in $\C$, { but it induces the sequence
$$\begin{smallmatrix}
4
\end{smallmatrix}\to \begin{smallmatrix}
0
\end{smallmatrix}\to \begin{smallmatrix}
1\\ 2
\end{smallmatrix}\to \begin{smallmatrix}
1
\end{smallmatrix}$$
which is not a weak kernel-cokernel sequence in $\C/\mathcal X$.
Therefore $\C/\mathcal X$ equipped with $\overline{\mathbb{E}}$ and
$\overline{\s}$ is not a $2$-exangulated category by \cite[Theorem 3.1]{HZZ2}}.
\end{example}


\begin{thebibliography}{99}

\bibitem{BBGH} R-L. Baillargeon, T. Br\"{u}stle, M. Gorsky, S. Hassoun. On the lattice of weakly exact structures. arXiv: 2009.10024v2, 2020.

\bibitem{B-TS} R. Bennett-Tennenhaus,  A. Shah. Transport of structure in higher homological algebra. J. Algebra 574: 514--549, 2021.

\bibitem{GKO} C. Geiss, B. Keller, S. Oppermann. $n$-angulated categories.  J. Reine Angew. Math.  675: 101--120, 2013.

\bibitem{HLN} M. Herschend, Y. Liu, H. Nakaoka.  $n$-exangulated categories (I): Definitions and fundamental properties. J. Algebra 570: 531--586, 2021.

\bibitem{HLN1} M. Herschend, Y. Liu, H. Nakaoka.  $n$-exangulated categories (II): Constructions from $n$-cluster tilting subcategories. J. Algebra 594: 636--684, 2022.


\bibitem{HZZ2} J. Hu, D. Zhang, P. Zhou. Two new classes of $n$-exangulated categories. J. Algebra 568: 1--21, 2021.
\bibitem{Ja} G. Jasso.  $n$-abelian and $n$-exact categories. Math. Z. 283(3--4): 703--759, 2016.

\bibitem{LZ} Y. Liu, P. Zhou. Frobenius $n$-exangulated categories. J. Algebra 559: 161--183, 2020.

\bibitem{HYA} H. Nakaoka, Y. Ogawa, A. Sakai. Locatization of extriangulated categories. arXiv: 2103.16907, 2021.



\bibitem{NP} H. Nakaoka, Y. Palu.  Extriangulated categories, Hovey twin cotorsion pairs and model structures. Cah. Topol. G\'{e}om. Diff\'{e}r. Cat\'{e}g. 60(2): 117--193, 2019.


\bibitem{ZW} Q. Zheng, J. Wei. $(n+2)$-angulated quotient categories. Algebra Colloq. 26(4): 689--720, 2019.
\end{thebibliography}
\end{document}